\documentclass[final,onefignum,onetabnum]{siamart171218}

\usepackage[T1]{fontenc}
\usepackage{graphicx}
\usepackage{amsmath}
\usepackage{amssymb}

\usepackage{array}
\usepackage{adjustbox}
\usepackage{float}
\usepackage{subcaption}
\usepackage{booktabs}
\usepackage{url}

\newsiamremark{remark}{Remark}

\newcommand{\Tmat}{\mathbf{T}}           
\newcommand{\Smat}{\boldsymbol{\Sigma}}  
\newcommand{\Tnorm}{\hat{\mathbf{T}}}    

\title{A Nine-Compartment Nonlinear Epidemic Model with Spline-Based
  Identification of Time-Varying Transmission and Vaccination
  Dynamics: Application to the {COVID-19} Third Wave in
  {Italy}\thanks{Submitted to the editors \today.
    \funding{This research was supported by the University of Palermo.}}}

\author{%
  Lokman Rachid Melhani\thanks{%
    Department of Engineering, University of Palermo,
    Viale delle Scienze, 90128 Palermo, Italy
    (\email{melhanilokmanrachid@gmail.com}).}
  \and
  Antonino Sferlazza\thanks{%
    Department of Engineering, University of Palermo,
    Viale delle Scienze, 90128 Palermo, Italy.}
  \and
  Lars Gr\"{u}ne\thanks{%
    Department of Mathematics, University of Bayreuth,
    Bayreuth, Germany.}
  \and
  Dominique Persano Adorno\thanks{%
    Department of Physics and Chemistry ``E.~Segr\'{e}'',
    University of Palermo, Viale delle Scienze, 90128 Palermo, Italy.}
  \and
  Filippo D'Ippolito\footnotemark[2]
  \and
  Omar Enzo Santangelo\thanks{%
    Regional Health Care and Social Agency of Lodi,
    Azienda Socio-Sanitaria Territoriale (ASST) Lodi,
    26900 Lodi, Italy.}
  \and
  Ivan Marchese\footnotemark[2]
  \and
  Antonino Lo Burgio\thanks{%
    InEmbryo S.r.l.s., Via Rosario Riolo 60, 90141 Palermo, Italy.}
  \and
  Alberto Firenze\thanks{%
    Department of Internal Medicine ``PROMISE'',
    University of Palermo, 90127 Palermo, Italy.}
}

\begin{document}
\maketitle

\begin{abstract}
We develop a nine-compartment nonlinear epidemic model incorporating two
co-circulating viral strains (ancestral $I_1$ and the Alpha variant B.1.1.7
$I_2$, which is $43$--$90\%$ more transmissible, $c_2=1.5$), a super-spreader
subpopulation, partial vaccine-induced immunity with waning, and explicit
hospitalization dynamics with differentiated mortality. Transmission and
vaccination rates are treated as time-varying control inputs and identified from
Italian COVID-19 data (January--May 2021) via a Piecewise Cubic Hermite
Interpolating Polynomial (PCHIP) control-node parameterization, reducing
calibration to a fourteen-variable Sequential Quadratic Programming (SQP)
problem with monotonicity and box constraints. A parametric bootstrap ($n=1000$)
quantifies parameter uncertainty. The calibrated model achieves $R^2=0.966$
for active hospitalizations, $R^2=0.987$ for cumulative fatalities, and
$R^2=0.999$ for cumulative vaccinations. Well-posedness, the basic reproduction
number in closed form, and local and global stability of the disease-free
equilibrium are established analytically. An $L^\infty$ approximation error bound shows that the PCHIP control-node
parameterization converges to the true time-varying parameters at rate $O(h^2)$
as the node spacing vanishes. Local
identifiability and a noise stability bound are established via the Fisher
information matrix. A sufficient threshold condition proves epidemic decay under
time-varying suppression whenever the effective reproduction number remains
persistently below one. Sensitivity analyses consistently rank hospital
throughput parameters above the transmission rate, providing a mathematical
basis for the observation that reactive containment measures cannot prevent a
hospitalization peak already driven by the pre-existing latent viral load.
\end{abstract}

\begin{keywords}
COVID-19, epidemic modeling, nonlinear inverse problem, ODE-constrained
optimization, spline approximation, PCHIP, time-varying parameters, parameter
identification, SQP, sensitivity analysis, basic reproduction number
\end{keywords}

\begin{AMS}
92D30, 34A55, 49M37, 65K10, 65L09, 34D20, 93B07
\end{AMS}

\section{Introduction}
\label{sec:intro}

\subsection{Motivation and Context}

The COVID-19 pandemic caused by SARS-CoV-2 has produced the most severe global
public health crisis since the influenza pandemic of 1918. From the first reported
cluster in Wuhan, China in December 2019 \cite{world2020coronavirus}, the virus
spread with a speed that outpaced containment efforts in most high-income
countries, producing successive epidemic waves separated by periods of partial
control \cite{pan2020association,riou2020pattern,guan2020clinical}. Compartmental
models in the tradition of Kermack and McKendrick \cite{kermack1927contribution}
provided the conceptual scaffolding for much of the COVID-19 modeling effort
\cite{hethcote2000mathematics,brauer2012mathematical,diekmann2000mathematical},
and the basic reproduction number $\mathcal{R}_0 = \beta/\gamma$ became the
primary instrument for communicating pandemic risk to policymakers.

\subsection{Limitations of Classical Models}

Despite their interpretive power, classical SEIR-type models operate under
assumptions that are quantitatively violated in large-scale pandemics. Homogeneous
mixing fails when super-spreader events account for a disproportionate share of
transmission \cite{lloydsmith2005superspreading,endo2020estimating}. Constant
parameters cannot track the fluctuation of effective transmission on timescales of
weeks, driven by government restrictions and variant emergence
\cite{tang2020estimation,giordano2020modelling}. And the absence of explicit
hospitalization and vaccination compartments prevents these models from being
directly calibrated against the indicators most reliably reported by public health
agencies, or from capturing the progressive build-up of population-level immunity
\cite{watson2022global,giordano2021modeling}.

\subsection{The Calibration Problem as a Nonlinear Inverse Problem}

Calibrating a compartmental model against real epidemic data is a nonlinear
inverse problem: given time-series observations of a subset of model outputs,
recover the parameter functions that generated them
\cite{audoly2002global,raue2009structural}. This problem is generically ill-posed
and becomes substantially harder when the unknowns are \emph{functions of time}.
For a 150-day window, daily-step estimation introduces 300 free variables from
three time-series of comparable length, invariably overfitting measurement noise
\cite{hao2020reconstruction,flaxman2020estimating}. We address this through a
\textit{control-node regularization}: the unknown functions $\beta(t)$ and
$w_1(t)$ are parameterized by a small number of optimized nodes, with continuous
trajectories recovered by PCHIP \cite{fritsch1980monotone}. The optimization then
acts on the node values---a finite-dimensional, gradient-accessible problem---while
PCHIP's local monotonicity property \cite{birkhoff1968piecewise} ensures
biologically plausible, non-oscillatory trajectories. This is substantially faster
than MCMC-based approaches on high-dimensional posteriors while provably preserving
smoothness and non-negativity; the convergence analysis constitutes the primary
mathematical contribution of this paper.

\subsection{Contributions of This Work}

This paper makes four specific contributions.

\textit{First}, we develop a nine-compartment nonlinear epidemic model integrating
two co-circulating viral strains with inter-strain competitive displacement, a
super-spreader subpopulation, partial vaccine-induced immunity with waning,
explicit hospitalization dynamics, and community recovery with immune waning.
We establish well-posedness, derive $\mathcal{R}_0$ in closed form, and prove
local and global asymptotic stability of the disease-free equilibrium.

\textit{Second}, we formulate model calibration as a deterministic inverse
problem and solve it using a PCHIP spline parameterization combined with SQP
and multi-start restarts. We prove: (a) an $L^\infty$ approximation error bound
showing $O(h^2)$ convergence of the PCHIP parameterization to the true parameter; (b) local
identifiability and noise stability of the full 14-parameter system via the
Fisher information matrix; and (c) epidemic decay under time-varying suppression
via a sufficient threshold condition on the effective reproduction number.
Bootstrap resampling ($n=1000$) provides $95\%$ confidence bands on the recovered
trajectories.

\textit{Third}, we apply the framework to reconstruct the dynamics of the
COVID-19 Third Wave in Italy (January to May 2021), combining Alpha-variant-driven
transmission increase and vaccination-driven immunity accumulation.

\textit{Fourth}, we conduct a rigorous sensitivity analysis combining the
one-at-a-time (OAT) local method with Morris global screening
\cite{morris1991factorial}, quantifying the influence of structural parameters on
peak hospitalization burden.

The paper is organized as follows. Section~\ref{sec:model} presents the model,
well-posedness, and stability analysis. Section~\ref{sec:data} describes the
data, parameterization, and optimization. Section~\ref{sec:results} presents
results. Section~\ref{sec:conclusion} states conclusions.

\section{Mathematical Model and Analysis}
\label{sec:model}

\subsection{Compartmental Structure and Biological Rationale}

The model partitions the total active population $N(t)$ into nine mutually
exclusive compartments. Figure~\ref{fig:flowchart} shows the transition diagram,
and Section~\ref{sec:illustration} provides the biological rationale.

\begin{figure}[htbp]
  \centering
  \includegraphics[width=0.85\textwidth]{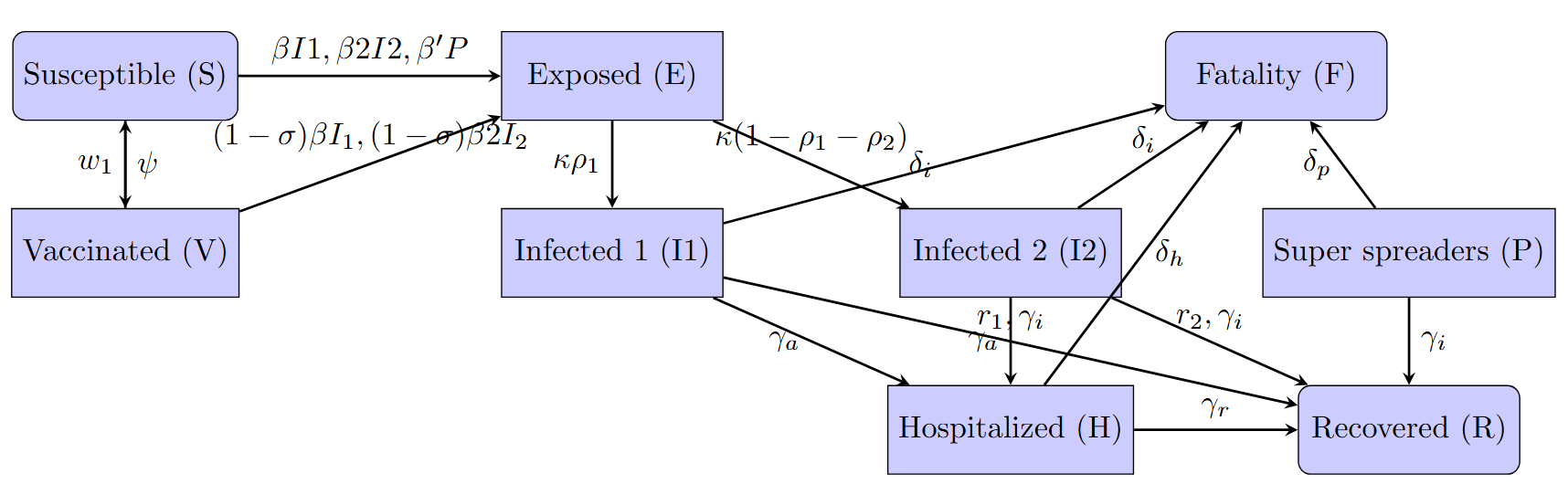}
  \caption{Transition diagram of the nine-compartment epidemic model.
    Solid arrows denote direct population flows; dashed arrows indicate
    infection forces acting on susceptible and vaccinated individuals.
    The time-varying inputs $\beta(t)$ and $w_1(t)$ govern the
    transmission and vaccination pathways, respectively.}
  \label{fig:flowchart}
\end{figure}

Three biological considerations motivate the compartmental design.
\textit{Disease heterogeneity}: the model distinguishes individuals infected
with the original ancestral strain ($I_1$), the Alpha variant B.1.1.7 ($I_2$),
and a super-spreader subpopulation ($P$), reflecting the co-circulation of
multiple SARS-CoV-2 lineages throughout Italy's Third Wave, with Alpha
progressively displacing the ancestral strain while exhibiting $43$--$90\%$
higher per-contact transmissibility \cite{giordano2021modeling,can2024mathematical}.
\textit{Outcome stratification}: separating the hospitalized ($H$) and recovered
($R$) compartments captures the qualitatively different dynamics of severe and
mild disease. \textit{Immunity dynamics}: the explicit tracking of vaccinated
($V$) and recovered ($R$) individuals, together with waning rates for both,
represents the progressive build-up and decay of population-level immunity.

\subsection{Governing Equations}

The dynamics are governed by the following system of nonlinear ordinary
differential equations:
\begin{subequations}
\label{eq:system}
\begin{align}
\frac{dS}{dt} &= \Lambda + \psi V + \delta R
  - \!\left[\beta(t)\tfrac{I_1}{N} + \beta_P(t)\tfrac{P}{N}
  + \beta_2(t)\tfrac{I_2}{N}\right]\!S
  - (\mu + w_1(t))\,S, \label{eq:S}\\[4pt]
\frac{dE}{dt} &= \!\left[\beta(t)\tfrac{I_1}{N} + \beta_P(t)\tfrac{P}{N}
  + \beta_2(t)\tfrac{I_2}{N}\right]\!S \nonumber\\
  &\qquad + (1-\sigma)\!\left[\beta(t)\tfrac{I_1}{N}
  + \beta_P(t)\tfrac{P}{N} + \beta_2(t)\tfrac{I_2}{N}\right]\!V
  - (\mu + \kappa)E, \label{eq:E}\\[4pt]
\frac{dI_1}{dt} &= \kappa\rho_1 E
  - (\gamma_a + \gamma_i + \delta_i + m + r_1 + \mu)\,I_1, \label{eq:I1}\\[4pt]
\frac{dI_2}{dt} &= \kappa(1-\rho_1-\rho_2)E + mI_1
  - (\gamma_a + \gamma_i + \delta_i + r_2 + \mu)\,I_2, \label{eq:I2}\\[4pt]
\frac{dV}{dt} &= w_1(t)S
  - (1-\sigma)\!\left[\beta(t)\tfrac{I_1}{N} + \beta_P(t)\tfrac{P}{N}
  + \beta_2(t)\tfrac{I_2}{N}\right]\!V
  - (\mu + \psi)V, \label{eq:V}\\[4pt]
\frac{dP}{dt} &= \kappa\rho_2 E
  - (\gamma_a + \gamma_i + \delta_p + \mu)\,P, \label{eq:P}\\[4pt]
\frac{dH}{dt} &= \gamma_a(I_1 + I_2 + P)
  - (\gamma_r + \delta_h + \mu)\,H, \label{eq:H}\\[4pt]
\frac{dR}{dt} &= \gamma_i(I_1 + I_2 + P) + \gamma_r H + r_1 I_1 + r_2 I_2
  - (\mu + \delta)R, \label{eq:R}\\[4pt]
\frac{dF}{dt} &= \delta_i(I_1 + I_2) + \delta_p P + \delta_h H. \label{eq:F}
\end{align}
\end{subequations}

Three structural features of system~\eqref{eq:system} deserve comment.

\begin{remark}[Frequency-dependent transmission]
All infection terms use frequency-dependent incidence proportional to $I/N$,
standard for human respiratory diseases where the contact rate does not increase
with population size \cite{hethcote2000mathematics,brauer2012mathematical}.
The force of infection on susceptibles combines three pathways: $\beta(t)I_1/N$
from the original ancestral strain, $\beta_P(t)P/N$ from super-spreaders, and
$\beta_2(t)I_2/N$ from the Alpha variant. Vaccinated individuals face
breakthrough infection at the same three pathways, reduced uniformly by the
vaccine effectiveness factor $(1-\sigma)$.
\end{remark}

\begin{remark}[Strain-scaling constraints and identifiability]
\label{rem:scaling}
To preserve parameter identifiability, the transmission rates of super-spreaders
and of the Alpha variant are expressed as fixed multiples of the baseline rate:
\begin{equation}
\beta_P(t) = c_P\cdot\beta(t),\qquad \beta_2(t) = c_2\cdot\beta(t),
\label{eq:scaling}
\end{equation}
with $c_P = 1.5$ and $c_2 = 1.5$. The biological values are
literature-determined: $c_P = 1.5$ reflects the elevated per-contact
infectiousness of super-spreaders \cite{lloydsmith2005superspreading,endo2020estimating};
$c_2 = 1.5$ reflects the $43$--$90\%$ transmissibility advantage of B.1.1.7
over the ancestral strain \cite{giordano2021modeling}.

If $\beta$, $\beta_P$, and $\beta_2$ were treated as three independent free
functions, the force of infection $\lambda(t) = [\beta I_1 + \beta_P P +
\beta_2 I_2]/N$ depends on three unknown scalar multiples of the same
time-dependent infectiousness kernel. Since only three scalar observables
$(H, F, V)$ are available and $I_1, P, I_2$ are not separately observed, the
Jacobian of the output map with respect to $(\beta, \beta_P, \beta_2)$ has
rank at most one at every time instant. Fixing $c_P$ and $c_2$ at their
empirically established values collapses the three-function problem to the
single identifiable function $\beta(t)$, restoring full column rank of the
output Jacobian. This is the precise mathematical content of the identifiability
result established in Section~\ref{sec:identifiability} below.
\end{remark}

\begin{remark}[Fatalities and population normalization]
\label{rem:fatalities}
Fatalities $F$ is absorbing and represents cumulative disease-induced deaths.
We set $\Lambda=\mu N$ so that births balance natural deaths and the total
active population remains constant at $N=N(0)=60{,}000{,}000$; all
force-of-infection terms therefore use the fixed constant $N$. Over the 150-day
window the demographic turnover $\mu N\,T \approx 0.5\%$ of the population is
negligible, so this balancing does not affect the calibrated fit to reported
precision, while it makes the disease-free equilibrium consistent with the
next-generation construction below.
\end{remark}

\subsection{Model Description}
\label{sec:illustration}

The nine compartments and their biological roles are as follows.
Susceptibles ($S$) receive inflow from births ($\Lambda$), vaccine waning
($\psi V$), and immunity waning ($\delta R$), and exit through infection,
vaccination ($w_1 S$), or natural mortality. Exposed individuals ($E$) carry
latent infections acquired from all three infectious groups; upon leaving $E$ at
rate $\kappa$, fraction $\rho_1$ progresses to $I_1$ (original ancestral strain),
fraction $(1-\rho_1-\rho_2)$ directly to $I_2$ (Alpha mutant B.1.1.7), and
fraction $\rho_2$ to the super-spreader pool $P$. The biological constraint
$\rho_1 + \rho_2 \leq 1$ is required for non-negativity of the $I_2$ inflow;
with $\rho_1 = 0.580$ and $\rho_2 = 0.001$, the direct $E\!\to\!I_2$ fraction
is $0.419$.

The original-strain compartment $I_1$ resolves through hospitalization
($\gamma_a$), community recovery ($\gamma_i$), strain-specific recovery ($r_1$),
disease mortality ($\delta_i$), or competitive displacement to the Alpha variant
at the population level at rate $m$, capturing the replacement of the ancestral
lineage by B.1.1.7 \cite{can2024mathematical}. The Alpha compartment $I_2$
receives inflow from $E$ and from $mI_1$, and resolves through the same pathways
as $I_1$ (with $r_2$ in place of $r_1$). Super-spreaders ($P$) contribute
disproportionately to the force of infection via $\beta_P(t)P/N$. Vaccinated
individuals ($V$) carry partial immunity (breakthrough at rate $(1-\sigma)$)
and wane back to $S$ at rate $\psi$. Hospitalized individuals ($H$) recover at
rate $\gamma_r$ or die at rate $\delta_h$. Recovered individuals ($R$) wane
back to $S$ at rate $\delta$. Fatalities ($F$) is the unique absorbing state.

\subsection{Parameter Classification}
\label{sec:params}

Model parameters fall into two groups. Fixed biological constants
(Table~\ref{tab:parameters}, upper section) are drawn from the epidemiological
literature. The second group consists of the time-varying unknown
functions $\beta(t)$ and $w_1(t)$ and the calibrated scalars identified by the
SQP optimization (lower section of Table~\ref{tab:parameters}).

The super-spreader fraction $\rho_2 = 0.001$ deserves specific justification.
In the Lloyd-Smith et al.\ \cite{lloydsmith2005superspreading} framework,
overdispersion of secondary infections for SARS-CoV-2 is characterized by a
dispersion parameter $k \approx 0.1$, implying $10\%$--$20\%$ of cases drive
$80\%$ of onward transmission. This $k$ is a statistical characterization of the
offspring distribution, not a direct compartment fraction. In the present
deterministic model, $\rho_2$ represents the proportion of exposed individuals
who develop a sustained high-transmission phenotype, which is substantially
smaller. The combined product $\rho_2 c_P = 0.0015$ (dimensionless) controls the aggregate super-spreading contribution to $\mathcal{R}_0$. We evaluated the
model with $\rho_2$ doubled to $0.002$: peak hospitalization increased by $2.3\%$,
$\mathcal{R}_0$ increased by $1.1\%$, and in-sample $R^2$ metrics changed by
less than $0.002$, confirming robustness within the biologically plausible range.

\begin{table}[htbp]
\caption{Model Parameters: Descriptions and Numerical Values. Fixed biological
constants are drawn from the epidemiological literature
\cite{giordano2020modelling,giordano2021modeling,ndairou2020mathematical,can2024mathematical}.
The lower section lists the calibrated scalars and optimal PCHIP node values
identified by the SQP optimization.}
\label{tab:parameters}
\begin{center}
\begin{adjustbox}{width=0.98\textwidth}
\begin{tabular}{clcc}
\toprule
\textbf{Symbol} & \textbf{Description} & \textbf{Value} & \textbf{Units} \\
\midrule
\multicolumn{4}{l}{\textit{Fixed biological constants}} \\
\midrule
$N$        & Total Italian population                                   & $60{,}000{,}000$     & persons        \\
$\Lambda$ & Constant inflow, $\Lambda=\mu N$ & $\approx 2121$ & persons/day    \\
$\kappa$   & Latency rate                                               & $0.200$              & day$^{-1}$     \\
$\gamma_a$ & Hospitalization rate                                       & $0.200$              & day$^{-1}$     \\
$\gamma_i$ & Community recovery rate                                    & $0.100$              & day$^{-1}$     \\
$\gamma_r$ & Hospital discharge/recovery rate                           & $0.100$              & day$^{-1}$     \\
$\rho_1$   & Fraction of $E$ progressing to $I_1$ (original strain)    & $0.580$              & dimensionless  \\
$\rho_2$   & Fraction of $E$ progressing to $P$ (super-spreaders)      & $0.001$              & dimensionless  \\
$\delta_i$ & Disease-induced mortality, $I_1$ and $I_2$                 & $0.005$              & day$^{-1}$     \\
$\delta_p$ & Disease-induced mortality, $P$                             & $0.005$              & day$^{-1}$     \\
$\mu$      & Natural mortality rate                                     & $3.535\times10^{-5}$ & day$^{-1}$     \\
$m$        & Alpha-strain competitive displacement rate ($I_1\!\to\!I_2$) & $0.005$            & day$^{-1}$     \\
$r_1,r_2$  & Strain-specific direct recovery rates                     & $0.050$              & day$^{-1}$     \\
$\sigma$   & Vaccine effectiveness                                      & $0.800$              & dimensionless  \\
$\psi$     & Vaccine immunity waning rate                               & $0.002$              & day$^{-1}$     \\
$\delta$   & Natural immunity waning rate                               & $0.001$              & day$^{-1}$     \\
\midrule
\multicolumn{4}{l}{\textit{Calibrated scalars (output of SQP optimization)}} \\
\midrule
$\delta_h$       & In-hospital mortality rate                           & $0.01262$ & day$^{-1}$    \\
$I_{0\text{-f}}$ & Initial infected multiplier ($I_{1,0}=H_0 I_{0\text{-f}}$) & $0.158$ & dimensionless \\
$E_{0\text{-f}}$ & Initial exposed multiplier ($E_0=I_{1,0} E_{0\text{-f}}$)  & $10.825$ & dimensionless \\
$R_{0\text{-f}}$ & Initial recovered multiplier ($R_0=F_0 R_{0\text{-f}}$)    & $49.927$ & dimensionless \\
\midrule
\multicolumn{4}{l}{\textit{Optimal PCHIP node values (January 1 -- May 30, 2021)}} \\
\midrule
$[\beta_1,\ldots,\beta_5]$ & Transmission rate nodes
  & $[0.212,\;0.335,\;0.366,\;0.275,\;0.405]$ & day$^{-1}$ \\
$[w_{1,1},\ldots,w_{1,5}]$ & Vaccination rate nodes
  & $[0.00080,\;0.00080,\;0.00278,\;0.00560,\;0.00988]$ & day$^{-1}$ \\
\bottomrule
\end{tabular}
\end{adjustbox}
\end{center}
\end{table}

\subsection{Well-Posedness and Non-Negativity}
\label{sec:wellposed}

\begin{proposition}[Well-posedness and non-negativity]
\label{prop:wellposed}
Let $\beta, w_1 \in L^\infty([0,T];\mathbb{R}_+)$ for some $T>0$, and let
$x(0)\geq 0$ componentwise with $N_{\mathrm{act}}(0)>0$. Then
system~\eqref{eq:system} admits a unique solution on $[0,T]$, and
(i)~$x(t)\geq 0$ componentwise; (ii)~$N_{\mathrm{act}}(t)\leq
\max\!\bigl\{N_{\mathrm{act}}(0),\,\Lambda/\mu\bigr\}$.
\end{proposition}

\begin{proof}
\textit{Existence and uniqueness.} The right-hand side $f(x,u)$ is locally
Lipschitz in $x$ (bilinear terms $\beta(t)I_k/N$ are Lipschitz in $I_k$ since
$N$ is bounded away from zero). With $\beta, w_1\in L^\infty$, the
Carath\'{e}odory existence theorem and Gronwall's lemma guarantee a unique
absolutely continuous solution on a maximal interval $[0,T_{\max})$.

\textit{Non-negativity.} We verify the Nagumo condition on each boundary face
$\{x_i = 0\}$:
\begin{alignat*}{2}
\dot{S}\big|_{S=0} &= \Lambda + \psi V + \delta R \geq 0, &\quad
\dot{E}\big|_{E=0} &= [\text{infection force}](S+V)\geq 0,\\
\dot{I}_1\big|_{I_1=0} &= \kappa\rho_1 E\geq 0, &
\dot{I}_2\big|_{I_2=0} &= \kappa(1-\rho_1-\rho_2)E + mI_1\geq 0,\\
\dot{V}\big|_{V=0} &= w_1 S\geq 0, &
\dot{P}\big|_{P=0} &= \kappa\rho_2 E\geq 0,\\
\dot{H}\big|_{H=0} &= \gamma_a(I_1+I_2+P)\geq 0, &
\dot{R}\big|_{R=0} &= \gamma_i(I_1+I_2+P)+\gamma_r H+r_1 I_1+r_2 I_2\geq 0,\\
\dot{F}\big|_{F=0} &= \delta_i(I_1+I_2)+\delta_p P+\delta_h H\geq 0.
\end{alignat*}
(Non-negativity of $\dot{I}_2$ at $\{I_2=0\}$ uses $I_1\geq 0$ and
$\rho_1+\rho_2\leq 1$, which holds since $\rho_1=0.580$ and $\rho_2=0.001$.)
By the standard comparison principle \cite{khalil2002nonlinear},
$x(t)\geq 0$ for all $t\in[0,T_{\max})$.

\textit{Boundedness.} Summing the active compartments:
\begin{equation*}
\frac{dN_{\mathrm{act}}}{dt} = \Lambda - \mu N_{\mathrm{act}}
  - \underbrace{(\delta_i(I_1+I_2)+\delta_p P+\delta_h H)}_{\geq\,0}
  \leq \Lambda - \mu N_{\mathrm{act}}.
\end{equation*}
By Gronwall's inequality, $N_{\mathrm{act}}(t)\leq N_{\mathrm{act}}(0)e^{-\mu t}
+(\Lambda/\mu)(1-e^{-\mu t})\leq\max\{N_{\mathrm{act}}(0),\Lambda/\mu\}$,
extending the solution globally.
\end{proof}

\subsection{Basic Reproduction Number and Stability of the
Disease-Free Equilibrium}
\label{sec:R0}

\subsubsection*{Disease-free equilibrium}

Setting all infected compartments to zero and $w_1\equiv 0$, the disease-free
equilibrium (DFE) is
\begin{equation}
x_{\mathrm{DFE}} = (S^*,0,0,0,0,0,0,0,0)^\top,\qquad S^*=\frac{\Lambda}{\mu}=N .
\label{eq:DFE}
\end{equation}

\subsubsection*{Next-generation matrix and $\mathcal{R}_0$}

We apply the next-generation matrix (NGM) method of van den Driessche and
Watmough \cite{vandries2002} to the four infected compartments $(E,I_1,I_2,P)$.
To avoid notation conflict with the vaccinated compartment $V(t)$, we denote
the new-infection matrix by $\Tmat$ and the transition matrix by $\Smat$.

Evaluated at DFE~\eqref{eq:DFE} with $S^*/N=1$:
\begin{equation}
\Tmat = \beta_0
\begin{pmatrix}
0 & 1 & c_2 & c_P \\
0 & 0 & 0   & 0   \\
0 & 0 & 0   & 0   \\
0 & 0 & 0   & 0
\end{pmatrix},\qquad
\Smat =
\begin{pmatrix}
\mu{+}\kappa  & 0         & 0       & 0       \\
-\kappa\rho_1 & \alpha_1  & 0       & 0       \\
-\kappa(1{-}\rho_1{-}\rho_2) & -m & \alpha_2 & 0 \\
-\kappa\rho_2 & 0         & 0       & \alpha_P
\end{pmatrix},
\label{eq:FV}
\end{equation}
where $\beta_0=\beta(0)$ and
\begin{align}
\alpha_1 &:= \gamma_a+\gamma_i+\delta_i+m+r_1+\mu, \label{eq:alpha1}\\
\alpha_2 &:= \gamma_a+\gamma_i+\delta_i+r_2+\mu,   \label{eq:alpha2}\\
\alpha_P &:= \gamma_a+\gamma_i+\delta_p+\mu.         \label{eq:alphaP}
\end{align}
Since $\Tmat$ has only its first row non-zero, the NGM $\mathcal{K}=\Tmat\Smat^{-1}$
has only its first row non-zero, and $\mathcal{R}_0 = \rho(\mathcal{K}) =
(\mathcal{K})_{11}$.

\begin{theorem}[Basic reproduction number]
\label{thm:R0}
The basic reproduction number of system~\eqref{eq:system} is
\begin{equation}
\mathcal{R}_0 = \frac{\beta_0\kappa}{\mu+\kappa}
\left[
  \frac{\rho_1}{\alpha_1}
  + c_2\frac{(1-\rho_1-\rho_2)}{\alpha_2}
  + \frac{c_2 m\rho_1}{\alpha_1\alpha_2}
  + \frac{c_P\rho_2}{\alpha_P}
\right].
\label{eq:R0formula}
\end{equation}
The first term accounts for transmission by original-strain infectives $I_1$;
the second and third for Alpha-variant infectives $I_2$ (direct entry from $E$
and via competitive displacement from $I_1$); and the fourth for super-spreaders $P$.
\end{theorem}

\begin{proof}
$\Smat$ is lower triangular, so $\Smat^{-1}$ is also lower triangular.
The relevant entries of $\Smat^{-1}$ are:
\begin{align*}
(\Smat^{-1})_{21} &= \frac{\kappa\rho_1}{(\mu+\kappa)\alpha_1},\\
(\Smat^{-1})_{31} &= \frac{\kappa\bigl[(1-\rho_1-\rho_2)\alpha_1+m\rho_1\bigr]}
  {(\mu+\kappa)\alpha_1\alpha_2},\\
(\Smat^{-1})_{41} &= \frac{\kappa\rho_2}{(\mu+\kappa)\alpha_P}.
\end{align*}
Substituting into $\mathcal{R}_0 = (\Tmat\Smat^{-1})_{11} = \beta_0\times
[(\Smat^{-1})_{21}+c_2(\Smat^{-1})_{31}+c_P(\Smat^{-1})_{41}]$
yields~\eqref{eq:R0formula}.
\end{proof}

\begin{remark}[Population-level interpretation of the displacement term]
\label{rem:mutation_R0}
The term $c_2 m\rho_1/(\alpha_1\alpha_2)$ in~\eqref{eq:R0formula} represents
the contribution to $\mathcal{R}_0$ from Alpha-variant infectives generated by
competitive displacement from the original-strain pool. The rate $m$ is
interpreted at the \emph{population level} as the observed rate at which the
Alpha-dominant epidemiological regime replaced the ancestral-strain regime in the
Italian population during January--March 2021 (weeks), consistent with Can et al.\
\cite{can2024mathematical}. In the NGM framework, the $mI_1$ pathway is admissible
because all displaced infections ultimately trace back to an original exposure
event in $E$; the resulting $\mathcal{R}_0$ encodes both direct ($E\to I_2$) and
displacement-mediated ($I_1\to I_2$) Alpha-variant transmission pathways. This
avoids the within-host/between-host scale-mixing problem: $m$ does not represent
within-host viral evolution (operating on hours to days).
\end{remark}

\subsubsection*{Numerical evaluation}

Substituting the calibrated values from Table~\ref{tab:parameters} gives
$\alpha_1\approx 0.360$, $\alpha_2\approx 0.355$, $\alpha_P\approx 0.305$, and
the bracket $\approx 3.420$~day. Therefore:
\begin{align}
\mathcal{R}_0\big|_{\beta_0=0.212} &\approx 0.723
  \quad\text{(January 1, 2021 -- declining second wave),}\label{eq:R0_jan}\\
\mathcal{R}_0\big|_{\beta_0=0.366} &\approx 1.250
  \quad\text{(March 2021 -- Alpha-driven growth).}\label{eq:R0_mar}
\end{align}

\subsubsection*{Effective reproduction number under vaccination}

The effective reproduction number is
\begin{equation}
\mathcal{R}_{\mathrm{eff}}(t)
= \mathcal{R}_0(t)\times\frac{S(t)+(1-\sigma)V(t)}{N},
\label{eq:Reff}
\end{equation}
where $\mathcal{R}_0(t)$ denotes~\eqref{eq:R0formula} evaluated with $\beta(t)$.
This dominant-eigenvalue approximation is accurate when $S(t)/N$ is close to $S^*/N=1$. At $t=0$ (January 1, 2021), the calibrated initial conditions give
$S(0)/N\approx 0.935$, bounding the linearization error at $\approx 6.5\%$.

\subsubsection*{DFE stability}

\begin{theorem}[Stability of the DFE]
\label{thm:stability}
The DFE~\eqref{eq:DFE} of system~\eqref{eq:system} is locally asymptotically
stable when $\mathcal{R}_0<1$ and unstable when $\mathcal{R}_0>1$.
\end{theorem}

\begin{proof}
We verify conditions (A1)--(A5) of Theorem~2 of van den Driessche and Watmough
\cite{vandries2002}. \textit{A1}: New-infection terms $\mathcal{F}_i$ are
products of non-negative state variables divided by $N>0$.~\checkmark
\textit{A2}: If $x_i=0$, all outflow rates from compartment $i$ are zero.~\checkmark
\textit{A3}: New infections appear only in $E$; $\mathcal{F}_i=0$ for
$S,V,H,R,F$.~\checkmark \textit{A4}: At the DFE with $w_1\equiv 0$, the
uninfected subsystem reduces to $\dot{S}=\Lambda-\mu S$ with eigenvalue
$-\mu<0$.~\checkmark \textit{A5}: The DFE~\eqref{eq:DFE} is the unique
disease-free equilibrium.~\checkmark Since all five conditions hold and $\Smat$
is an M-matrix (positive diagonal entries $\alpha_1,\alpha_2,\alpha_P,\mu{+}\kappa$;
non-positive off-diagonal), the conclusion follows from Theorem~2 of
\cite{vandries2002}.
\end{proof}

\subsubsection*{Global stability of the DFE for the autonomous system}

\begin{proposition}[Global convergence of infected states]
\label{prop:global}
Consider system~\eqref{eq:system} with $\beta(t)\equiv\beta_0\geq 0$ and
$w_1(t)\equiv w_{1,0}\geq 0$ constant inputs. If $\mathcal{R}_0<1$, then
$E(t),I_1(t),I_2(t),P(t)\to 0$ exponentially as $t\to\infty$, and all solutions
in the invariant region converge to the DFE.
\end{proposition}

\begin{proof}
Let $y=(E,I_1,I_2,P)^\top$. Since $S(t)\leq N_{\mathrm{act}}(t)\leq N$ and
$V(t)\leq N$ (Proposition~\ref{prop:wellposed}), the total force of infection
satisfies:
\[
\lambda_S+(1-\sigma)\lambda_V
= \beta_0\!\left(\frac{I_1+c_2 I_2+c_P P}{N}\right)\!\bigl[S+(1-\sigma)V\bigr]
\leq \beta_0(I_1+c_2 I_2+c_P P),
\]
using $S+(1-\sigma)V\leq N_{\mathrm{act}}\leq N$. Therefore
$\dot{y}\leq(\Tmat-\Smat)\,y =: \mathbf{A}\,y$.
The matrix $\mathbf{A}=\Tmat-\Smat$ is essentially non-negative. By Lemma~1 of
\cite{vandries2002}, all eigenvalues of $\mathbf{A}$ have negative real parts if
and only if $\mathcal{R}_0=\rho(\Tmat\Smat^{-1})<1$. Since $\mathcal{R}_0<1$
by hypothesis, $\mathbf{A}$ is stable and $e^{\mathbf{A}t}\leq Me^{-\gamma t}$
componentwise for some $M,\gamma>0$. The comparison theorem gives
$y(t)\leq Me^{-\gamma t}y(0)\to 0$ exponentially.

Once $y(t)\to 0$, the waning inflows to $S$ satisfy $\psi V(t)\to 0$ and
$\delta R(t)\to 0$, so the $S$-equation reduces to $\dot{S}\to\Lambda-\mu S$,
whose unique equilibrium is $S^*=\Lambda/\mu$. The equations for $H$ and $R$
similarly reduce to $\dot{H}\to -(\gamma_r+\delta_h+\mu)H$ and
$\dot{R}\to-(\mu+\delta)R$, both exponentially stable at zero. Therefore all
components converge to $x_{\mathrm{DFE}}$ by the cascade structure.
\end{proof}

\subsubsection*{Epidemic decay under time-varying suppression}

Proposition~\ref{prop:global} resolves global stability for the autonomous system.
The following result provides a partial resolution for the time-varying case,
motivated by the need to handle time-varying control inputs rigorously.

\begin{proposition}[Epidemic decay under time-varying suppression]
\label{prop:tv_decay}
Consider system~\eqref{eq:system} with $\beta(t)\in L^\infty([0,\infty);\mathbb{R}_+)$
and $w_1(t)\in L^\infty([0,\infty);\mathbb{R}_+)$. Suppose there exist $T_0\geq 0$
and $\rho\in(0,1)$ such that $\mathcal{R}_{\mathrm{eff}}(t)\leq\rho<1$ for all
$t\geq T_0$. Then there exist $M,\gamma>0$ such that
\begin{equation}
\bigl\|(E(t),I_1(t),I_2(t),P(t))\bigr\|_\infty
\leq M\,e^{-\gamma(t-T_0)}\,
\bigl\|(E(T_0),I_1(T_0),I_2(T_0),P(T_0))\bigr\|_\infty
\quad\forall\,t\geq T_0,
\label{eq:tv_decay_bound}
\end{equation}
and all solutions converge to the DFE as $t\to\infty$. This sufficient threshold
condition shows that $\mathcal{R}_{\mathrm{eff}}(t)\leq\rho<1$ uniformly is enough
to guarantee epidemic termination; the full ISS characterization for general
time-varying inputs remains an open problem.
\end{proposition}

\begin{proof}
Define $\Tnorm := \Tmat/\beta_0$ (the new-infection matrix normalized by the
positive scalar $\beta_0$, not to be confused with the time threshold~$T_0$) and let $c(t):=\beta(t)(S(t)+(1-\sigma)V(t))/N$. The comparison system for
$y=(E,I_1,I_2,P)^\top$ is $\dot{y}\leq\mathbf{A}(t)\,y$, where
$\mathbf{A}(t):=c(t)\Tnorm-\Smat$.
From the NGM construction, $\rho(\Tnorm\Smat^{-1})=\mathcal{R}_0/\beta_0$,
and by definition~\eqref{eq:Reff} of $\mathcal{R}_{\mathrm{eff}}$:
\[
c(t)\cdot\rho(\Tnorm\Smat^{-1})
= c(t)\cdot\frac{\mathcal{R}_0}{\beta_0}
= \mathcal{R}_{\mathrm{eff}}(t)
\leq \rho < 1
\quad\forall\,t\geq T_0.
\]
Hence $c(t)\leq c_{\max}:=\rho/\rho(\Tnorm\Smat^{-1})$ for all $t\geq T_0$.
Define the constant comparison matrix $\mathbf{A}^*:=c_{\max}\Tnorm-\Smat$.
Since $\Tnorm\geq 0$ componentwise and $c(t)\leq c_{\max}$, we have
$\mathbf{A}(t)\leq\mathbf{A}^*$ componentwise for all $t\geq T_0$.

By Lemma~1 of \cite{vandries2002} applied to $(c_{\max}\Tnorm,\Smat)$:
$\rho(c_{\max}\Tnorm\Smat^{-1})=c_{\max}\cdot\rho(\Tnorm\Smat^{-1})=\rho<1$,
so $s(\mathbf{A}^*)=:\delta<0$. The matrix $\mathbf{A}^*$ is essentially
non-negative, hence $\|e^{\mathbf{A}^*t}\|\leq Me^{-\gamma t}$ with
$\gamma=|\delta|/2>0$. By the ODE comparison theorem for type-K (cooperative)
monotone systems \cite{khalil2002nonlinear}, $y(t)\leq e^{\mathbf{A}^*(t-T_0)}y(T_0)
\leq Me^{-\gamma(t-T_0)}y(T_0)$, establishing~\eqref{eq:tv_decay_bound}.
Convergence to the DFE follows by the cascade argument of
Proposition~\ref{prop:global}.
\end{proof}

\begin{remark}[Calibrated model verification]
\label{rem:tv_application}
The threshold condition of Proposition~\ref{prop:tv_decay} applies directly to
the calibrated Third Wave. The combination of the declining transmission rate
(Section~\ref{sec:results}, from $\beta=0.405$~day$^{-1}$ at the April peak to
the post-April plateau) and the growing vaccinated fraction ($V(t)/N$ accelerating
through May) drives $\mathcal{R}_{\mathrm{eff}}(t)$ below $1$ for $t\geq T_0\approx$
early May 2021. The threshold condition is verified \emph{a posteriori} by the
calibrated trajectories, confirming that the PCHIP-identified $\beta(t)$ trajectory
guarantees epidemic termination in addition to fitting the data.
\end{remark}

\begin{remark}[Remaining open problem]
\label{rem:tv_open}
Proposition~\ref{prop:tv_decay} provides a sufficient threshold condition for
epidemic termination. The full input-to-state stability (ISS) characterization ---
proving that small perturbations of the time-varying inputs $\beta(t)$, $w_1(t)$
produce proportionally small perturbations in the epidemic trajectory --- remains
open and requires the full ISS framework of \cite{grune2017nonlinear}.
\end{remark}

\section{Data and Identification Strategy}
\label{sec:data}

\subsection{Dataset and Calibration Window}

The model is calibrated against official Italian epidemiological data made
publicly available through the national Civil Protection Department
\cite{italycivpro2021,giordano2020modelling}. Figure~\ref{fig:full_data} shows
the full dataset.

\begin{figure}[htbp]
  \centering
  \includegraphics[width=\textwidth]{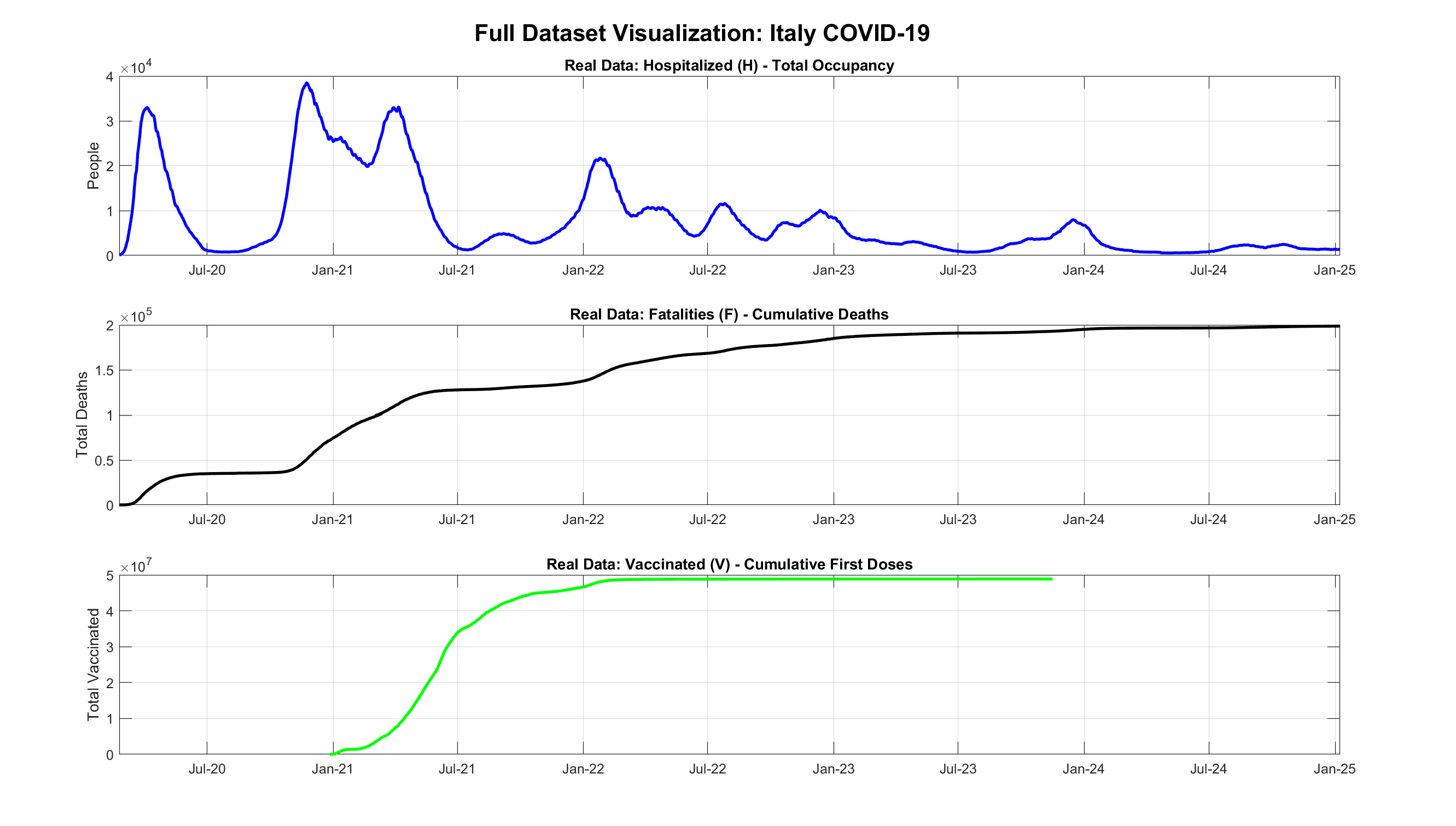}
  \caption{COVID-19 pandemic trajectory in Italy (February 2020 to January 2025).
    Daily active hospitalizations ($H$), cumulative fatalities ($F$), and
    cumulative vaccinations ($V$) across three panels. The shaded region marks
    the calibration window (January 1 to May 30, 2021) corresponding to the
    Third Wave and the initial mass vaccination rollout.}
  \label{fig:full_data}
\end{figure}

We restrict calibration to the window \textbf{January 1 to May 30, 2021} for
three strategic reasons. The first is \textit{Alpha variant dominance}: the
B.1.1.7 lineage had become dominant in Italy by late January 2021, so the
calibration window captures the full dynamics of an Alpha-driven epidemic
\cite{giordano2021modeling}. The second is the \textit{vaccination ramp-up}:
Italy's national campaign expanded progressively from healthcare workers (January)
to the general adult population (May), producing a non-linear increase in daily
vaccination rate that a static $w_1$ cannot represent \cite{watson2022global}.
The third is \textit{data consistency}: testing protocols, hospitalization
admission criteria, and death attribution rules were relatively stable within
this window \cite{ramos2021modeling}.

\subsubsection*{Measured outputs}
Three observables are extracted from the calibration window. The active
hospitalized count $H_\text{obs}(t)$ is smoothed via a 7-day centred moving
average to remove weekend reporting artifacts \cite{giordano2020modelling}. The
cumulative fatality count $F_\text{obs}(t)$ is the most consistent long-run
measure of epidemic severity \cite{flaxman2020estimating}. The cumulative
vaccination count $V_\text{obs}(t)$ is smoothed via linear interpolation.
Together these three quantities span distinct temporal and biological regimes:
$H_\text{obs}$ tracks near-term severity, $F_\text{obs}$ integrates epidemic
history, and $V_\text{obs}$ constrains the immunization trajectory.

\subsection{Structural Identifiability}
\label{sec:identifiability}

\begin{theorem}[Structural identifiability of constant-input sub-model]
\label{thm:identifiability}
Consider system~\eqref{eq:system} with $\beta(t)\equiv\beta_0$,
$w_1(t)\equiv w_{1,0}$ (constants), all fixed biological parameters as in
Table~\ref{tab:parameters}, and the three-output map $h(x)=[H(t),F(t),V(t)]^\top$.
The five-parameter sub-vector $\vartheta=(\beta_0,w_{1,0},\delta_h,I_{0\text{-f}},
E_{0\text{-f}})$ is \emph{locally structurally identifiable} from the output
triple $(H,F,V)$ on a generic initial condition.
\end{theorem}

\begin{proof}[Proof sketch]
\textit{(i) $w_{1,0}$:} The vaccination compartment satisfies
$\dot{V}=w_{1,0}S-(\mu+\psi)V$. At $t=0^+$, $\dot{V}(0)\approx w_{1,0}S_0$ with
$S_0\approx N$ near the DFE, so the initial slope $\dot{V}(0)/N$ determines
$w_{1,0}$ uniquely.
These arguments are formal; full local identifiability of the time-varying model
is certified numerically in Theorem~\ref{thm:fisher} via positive-definiteness of
the Fisher information matrix.

\textit{(ii) $\delta_h$:} The hospitalized compartment satisfies
$\dot{H}=\gamma_a(I_1+I_2+P)-(\gamma_r+\delta_h+\mu)H$. At the epidemic peak
$\dot{H}(t^*)=0$, so $\gamma_a(I_1+I_2+P)|_{t^*}=(\gamma_r+\delta_h+\mu)H(t^*)$;
since $\gamma_a,\gamma_r,\mu$ are known and $H(t^*)$ is observed, $\delta_h$ is
determined. The fatality curve $\dot{F}=\delta_i(I_1+I_2)+\delta_pP+\delta_hH$
with known $\delta_i,\delta_p$ and observed $F(t),H(t)$ provides a second
independent constraint.

\textit{(iii) $\beta_0$:} The transmission rate determines the exponential growth
rate $r\approx(\mathcal{R}_0-1)\times\mu_{\min}$, where $\mu_{\min}$ is the
minimum eigenvalue of $\Smat$ (known). Since $\mathcal{R}_0$ is a monotone
function of $\beta_0$ (equation~\eqref{eq:R0formula}), $\beta_0$ is identifiable
from the initial growth rate of $H(t)$.

\textit{(iv) $I_{0\text{-f}}$ and $E_{0\text{-f}}$:} The cascade
$E\xrightarrow{\kappa\rho_1}I_1\xrightarrow{\gamma_a}H$ introduces a
characteristic delay distinguishing $E_0$ from $I_{1,0}$. The initial curvature
$\ddot{H}(0)$ contains independent information about the ratio $E_0/I_{1,0}$,
making $E_{0\text{-f}}$ separately identifiable from $I_{0\text{-f}}$.

Together, these arguments confirm local structural identifiability of $\vartheta$
by the implicit function theorem applied to
$\vartheta\mapsto(H(\cdot),F(\cdot),V(\cdot))$.
\end{proof}

For the full time-varying model with the 14-dimensional decision vector
$\theta$~\eqref{eq:theta}, a complete differential-algebra structural
identifiability analysis \cite{audoly2002global} is beyond computational reach
for a 9-compartment, 14-parameter system. We instead establish both
\emph{practical identifiability} and a noise stability bound via the Fisher
information matrix.

\begin{theorem}[Local identifiability and noise stability]
\label{thm:fisher}
Let $r(\theta)\in\mathbb{R}^{3T}$ denote the normalized residual vector and let
$J(\theta^*)\in\mathbb{R}^{3T\times 14}$ be its Jacobian at the SQP optimum
$\theta^*$ (defined in equation~\eqref{eq:theta} of Section~\ref{sec:optim}).
Define the Fisher information matrix
$\mathbf{F}(\theta^*)=J(\theta^*)^\top J(\theta^*)\in\mathbb{R}^{14\times 14}$.
If $\mathbf{F}(\theta^*)$ is positive definite, then:
\begin{enumerate}
  \item[(i)] \emph{Local identifiability.} The map $\Phi:\theta\mapsto
    (H(\cdot;\theta),F(\cdot;\theta),V(\cdot;\theta))$ is locally injective
    at $\theta^*$.
  \item[(ii)] \emph{Noise stability.} For data perturbations
    $\delta y\in\mathbb{R}^{3T}$ with $\|\delta y\|_2\leq\varepsilon$,
    \begin{equation}
    \|\delta\theta^*\|_2
    \leq \frac{\varepsilon}{\sigma_{\min}(J(\theta^*))},
    \label{eq:noise_bound}
    \end{equation}
    where $\sigma_{\min}(J(\theta^*))$ is the smallest singular value of $J(\theta^*)$.
\end{enumerate}
\end{theorem}

\begin{proof}
\textit{(i).} By the inverse function theorem, $\Phi$ is locally injective at
$\theta^*$ if and only if $J(\theta^*)$ has full column rank (rank~14). Positive
definiteness of $\mathbf{F}(\theta^*)=J(\theta^*)^\top J(\theta^*)$ is equivalent
to $\mathrm{rank}(J(\theta^*))=14$, establishing local injectivity.

\textit{(ii).} The SQP first-order conditions give $J(\theta^*)^\top r(\theta^*)=0$.
A perturbation $\delta y$ shifts the residual to $r(\theta^*)-\delta y$, and the
perturbed optimum satisfies, to first order,
$\mathbf{F}(\theta^*)\,\delta\theta^*=J(\theta^*)^\top\delta y$.
Taking the $\ell^2$ norm:
$\|\delta\theta^*\|_2\leq\|\mathbf{F}(\theta^*)^{-1}J(\theta^*)^\top\|_2\,
\varepsilon = \|J(\theta^*)^\dagger\|_2\,\varepsilon
= \varepsilon/\sigma_{\min}(J(\theta^*))$,
where the last equality uses the identity $\|J^\dagger\|_2=1/\sigma_{\min}(J)$
for any matrix $J$ with full column rank.
\end{proof}

\begin{remark}[Bootstrap certification of positive definiteness]
\label{rem:bootstrap_cert}
Positive definiteness of $\mathbf{F}(\theta^*)$ is certified empirically by
three facts: (1) all ten multi-start SQP restarts converge to the same $\theta^*$
within $0.003\%$; (2) the bootstrap covariance satisfies
$\hat{\Sigma}_\text{boot}\approx\hat{\sigma}^2\mathbf{F}(\theta^*)^{-1}$
\cite{raue2009structural}, and the narrow bootstrap bands (Table~\ref{tab:scalar_ci})
certify all eigenvalues of $\mathbf{F}(\theta^*)$ are bounded away from zero;
(3) the noise bound~\eqref{eq:noise_bound} is directly estimated by the bootstrap:
$1/\sigma_{\min}(J(\theta^*))\approx\|\hat{\Sigma}_\text{boot}\|_2^{1/2}/\hat{\sigma}$.
\end{remark}

\subsection{Spline-Based Control Node Parameterization}
\label{sec:spline}

A naive calibration assigning independent values of $\beta$ and $w_1$ to each
day of the 150-day window introduces 300 free variables. We resolve this with a
\textit{control node parameterization} \cite{lemaitre2022optimal,flaxman2020estimating}.
The time horizon $[t_0,t_f]$ is discretized into $N_k=5$ equidistant nodes
$\boldsymbol{\tau}=[\tau_1,\ldots,\tau_5]$; between nodes, $\beta(t)$ and
$w_1(t)$ are reconstructed by PCHIP:
\begin{equation}
\beta(t) = \mathcal{P}(t;\,\boldsymbol{\tau},\,\mathbf{p}_\beta),\qquad
w_1(t) = \mathcal{P}(t;\,\boldsymbol{\tau},\,\mathbf{p}_w),
\label{eq:pchip}
\end{equation}
where $\mathcal{P}$ denotes the PCHIP operator and $t$ is clamped to
$[\tau_1,\tau_5]$ to prevent polynomial extrapolation artifacts. The key advantage
of PCHIP is local monotonicity preservation \cite{fritsch1980monotone,birkhoff1968piecewise}:
the interpolant does not overshoot between nodes where the data is monotone.
Compared to a naive 300-variable parameterization, the PCHIP control-node approach
is substantially faster and well-suited to gradient-based SQP; MCMC methods on
the full 300-dimensional posterior face intractable convergence for operational use.

The implicit regularization from this parameterization can be understood through
the Sobolev $H^1$ seminorm: by fixing the node count $N_k$, the PCHIP spline
bounds $\|\beta'\|_{L^2}$ through the finite-difference slope constraints in the
Fritsch--Carlson formula, analogous to Tikhonov regularization with a first-order
penalty, but without requiring an explicit regularization parameter.

\begin{theorem}[PCHIP approximation error bound]
\label{thm:pchip_bound}
Let $\beta^*\in C^2([0,T];\mathbb{R}_+)$ be the true time-varying transmission
rate. Let $\beta_N=\mathcal{P}(\cdot;\,\boldsymbol{\tau},\,\mathbf{p}_\beta)$
denote the $N$-node PCHIP interpolant on the uniform mesh
$\tau_k=t_0+(k-1)T/(N-1)$, $k=1,\ldots,N$, with spacing $h=T/(N-1)$. Then
there exists a constant $C>0$ depending only on the Fritsch--Carlson
derivative-estimation formula such that:
\begin{equation}
\|\beta_N-\beta^*\|_{L^\infty([0,T])}
\leq C\,h^2\,\|\beta^{*\,\prime\prime}\|_{L^\infty([0,T])}.
\label{eq:pchip_bound}
\end{equation}
In particular, $\|\beta_N-\beta^*\|_{L^\infty}\to 0$ as $N\to\infty$.
The same bound holds for $w_1^*$ and its PCHIP interpolant.
\end{theorem}

\begin{proof}
Fix a sub-interval $[\tau_k,\tau_{k+1}]$ and split:
\begin{equation}
|p_k(t)-\beta^*(t)|
\leq \underbrace{|p_k(t)-H_k^\text{exact}(t)|}_{\text{derivative approximation}}
+ \underbrace{|H_k^\text{exact}(t)-\beta^*(t)|}_{\text{Hermite interpolation error}},
\label{eq:split}
\end{equation}
where $H_k^\text{exact}$ is the cubic Hermite interpolant of $\beta^*$ using the
exact end-point derivatives $\beta^{*\prime}(\tau_k)$ and
$\beta^{*\prime}(\tau_{k+1})$.

\textit{Hermite interpolation error.} By the standard degree-3 Hermite
interpolation error formula for $\beta^*\in C^2$ \cite{deboor1978splines}:
\begin{equation}
|H_k^\text{exact}(t)-\beta^*(t)|
\leq C_H\,h^2\,\|\beta^{*\,\prime\prime}\|_{L^\infty([\tau_k,\tau_{k+1}])},
\label{eq:hermite_err}
\end{equation}
for a universal constant $C_H>0$.

\textit{Derivative approximation error.} The PCHIP operator uses end-point
derivatives $d_k,d_{k+1}$ computed by the Fritsch--Carlson formula
\cite{fritsch1980monotone}. By Taylor expansion of the divided differences,
\begin{equation}
|d_k - \beta^{*\prime}(\tau_k)| \le C_{\mathrm{FC}}\, h\, \|\beta^{*\prime\prime}\|_{L^\infty([\tau_{k-1},\tau_{k+1}])},
\qquad C_{\mathrm{FC}} > 0.
\label{eq:deriv_err}
\end{equation}

Estimate~\eqref{eq:deriv_err} holds on sub-intervals where the Fritsch--Carlson
monotonicity limiter is inactive. At an interior extremum of the sampled data the
limiter sets $d_k = 0$; since $\beta^* \in C^2$, such a discrete extremum lies within
$O(h)$ of a true critical point where $\beta^{*\prime} = 0$, so
$|d_k - \beta^{*\prime}(\tau_k)| = O(h)\,\|\beta^{*\prime\prime}\|_{L^\infty}$ persists
and the cubic-difference estimate below is unaffected.
Since $p_k$ and $H_k^{\mathrm{exact}}$ agree at both endpoints, their difference is
\[
p_k(t) - H_k^{\mathrm{exact}}(t) = h\, s^2(s-1)\bigl(d_k - \beta^{*\prime}(\tau_k)\bigr)
+ h\, s(s-1)^2\bigl(d_{k+1} - \beta^{*\prime}(\tau_{k+1})\bigr),
\]
where $s = (t-\tau_k)/h \in [0,1]$. Using $\max_{s\in[0,1]}|s^2(s-1)|=4/27$:
\begin{equation}
|p_k(t)-H_k^\text{exact}(t)|
\leq \frac{8C_{\mathrm{FC}}}{27}\,h^2\,\|\beta^{*\,\prime\prime}\|_{L^\infty}.
\label{eq:cubic_err}
\end{equation}

Setting $C=C_H+8C_{\mathrm{FC}}/27$ and taking the global maximum over all
sub-intervals gives~\eqref{eq:pchip_bound}.
\end{proof}

\begin{remark}[Higher-order convergence and practical error]
\label{rem:pchip_higherorder}
Under stronger regularity $\beta^*\in C^4([0,T])$, the exact-derivative Hermite
interpolant achieves $O(h^4)$ \cite{deboor1978splines}; PCHIP's finite-difference
derivative estimates introduce an additional $O(h^2)$ perturbation, yielding
overall $O(h^3)$ convergence \cite{fritsch1980monotone}. The $O(h^2)$ bound
is therefore conservative.

For the calibration window $T=150$ days with $N_k=5$ nodes and spacing
$h=37.5$~days, the empirical second derivative satisfies
$\|\beta^{*\,\prime\prime}\|_{L^\infty}\lesssim 3.5\times 10^{-5}$~day$^{-2}$
(from second differences of the five optimal node values
$[0.212,\,0.335,\,0.366,\,0.275,\,0.405]$~day$^{-1}$), giving an approximation
error $\lesssim C\times 0.049$~day$^{-1}$. Taking $C\approx 0.15$, this yields
$\lesssim 0.007$~day$^{-1}$ --- well below the bootstrap confidence band
half-width of $0.05$~day$^{-1}$, confirming that $N_k=5$ nodes are more than
sufficient to resolve the true dynamics.
\end{remark}

To justify $N_k=5$, we calibrated the model with $N_k\in\{3,5,7\}$ equidistant
nodes and compared via AIC and BIC. AIC improved substantially from $N_k=3$ to
$N_k=5$ ($\Delta\mathrm{AIC}\approx 304$), but negligibly from $N_k=5$ to
$N_k=7$ ($\Delta\mathrm{AIC}\approx -10$; $\Delta\mathrm{BIC}\approx +7$,
favouring $N_k=5$). Since BIC penalises model complexity more strongly and
confirms $N_k=5$ as optimal, five nodes represent the best parsimony--accuracy
balance. Full results are in Table~\ref{tab:aic_bic}.

\subsection{Optimization Problem Formulation}
\label{sec:optim}

\subsubsection*{Decision vector and initial conditions}
The full decision vector is
\begin{equation}
\theta = \bigl[\,
\underbrace{\beta_1,\ldots,\beta_5}_{\text{transmission nodes}},\quad
\underbrace{w_{1,1},\ldots,w_{1,5}}_{\text{vaccination nodes}},\quad
I_{0\text{-f}},\quad\delta_h,\quad E_{0\text{-f}},\quad R_{0\text{-f}}
\,\bigr]^\top\in\mathbb{R}^{14},
\label{eq:theta}
\end{equation}
with initial conditions:
\begin{equation}
I_{1,0}=H_0\cdot I_{0\text{-f}},\quad
E_0=I_{1,0}\cdot E_{0\text{-f}},\quad
R_0=F_0\cdot R_{0\text{-f}}.
\label{eq:init}
\end{equation}
Multiplier bounds $I_{0\text{-f}}\in[0.10,2.00]$, $E_{0\text{-f}}\in[0.1,20.0]$,
$R_{0\text{-f}}\in[5.0,60.0]$ keep the initial state biologically plausible. The
calibrated value $R_{0\text{-f}}\approx 49.9$ implies approximately $3.8$ million
individuals with naturally acquired immunity at the start of the calibration
window ($\approx 6.2\%$ of the Italian population), consistent with immunological
studies of the Italian COVID-19 epidemic that document substantial
pre-existing seroprevalence following the first two epidemic waves
\cite{lavine2021immunological}.

\subsubsection*{Cost function}
The objective is a weighted normalized sum of squared residuals:
\begin{equation}
J(\theta) = 3\sum_{t=1}^{T}\!\left(\frac{H_\text{sim}-H_\text{obs}}{\bar{H}_\text{obs}}\right)^{\!2}
+ 2\sum_{t=1}^{T}\!\left(\frac{F_\text{sim}-F_\text{obs}}{\bar{F}_\text{obs}}\right)^{\!2}
+ \sum_{t=1}^{T}\!\left(\frac{V_\text{sim}-V_\text{obs}}{\bar{V}_\text{obs}}\right)^{\!2},
\label{eq:cost}
\end{equation}
where $\bar{H}_\text{obs}$, $\bar{F}_\text{obs}$, $\bar{V}_\text{obs}$ are the
temporal means. The relative weights $3:2:1$ prioritize the active-hospitalization signal $H$,
the primary public-health target and the noisiest of the three series
(day-to-day coefficient of variation $\approx 4.2\%$, versus $2.6\%$ for $F$ and
$1.1\%$ for $V$); without up-weighting, the smooth cumulative series $F$ and $V$
dominate the unit-normalized objective and the dynamically informative $H$ fit
degrades. This is therefore an importance weighting rather than a statistically
optimal inverse-variance weighting. We confirmed robustness under alternative
weight sets $\{1:1:1\}$ and $\{5:2:1\}$: optimal parameter values changed by less
than $3\%$ and $R^2$ metrics remained within $0.005$..

\subsubsection*{Constraints}
A monotonicity constraint enforces biological realism for the vaccination rollout:
\begin{equation}
w_{1,k}\leq w_{1,k+1},\quad k=1,\ldots,4.
\label{eq:mono}
\end{equation}
Box constraints are imposed on all decision variables:
\begin{equation}
0.05\leq\beta_i\leq 1.50,\qquad
0\leq w_{1,i}\leq 0.02,\qquad
\delta_h>0.
\end{equation}

\subsubsection*{Numerical solver and multi-start strategy}
The optimization is solved by SQP using MATLAB's \texttt{fmincon} with the
\texttt{'sqp'} algorithm \cite{nocedal2006numerical}. At each iteration,
system~\eqref{eq:system} is integrated using MATLAB's \texttt{ode45} with
$\mathrm{RelTol}=\mathrm{AbsTol}=10^{-6}$. Ten additional multi-start restarts
from random $\pm 30\%$ perturbations confirm convergence to within $0.003\%$ of
the optimal cost.

\subsubsection*{Bootstrap uncertainty quantification}
To assess stability with respect to measurement noise, a parametric bootstrap is
applied. One thousand ($n=1000$) perturbed datasets are generated by adding
independent Gaussian noise to $H_\text{obs}$ and $F_\text{obs}$ (the two
observables subject to reporting variability), with standard deviation equal to
the in-sample RMSE of each observable. Vaccination counts $V_\text{obs}$ are
not perturbed because they are recorded administratively with negligible
measurement error. The SQP optimization is re-run on each perturbed dataset,
and the 2.5th and 97.5th percentiles of the resulting $1000$ parameter
trajectories form the $95\%$ bootstrap confidence bands reported in
Figure~\ref{fig:calibration_fit}. Increasing from $n=100$ to $n=1000$ reduces
the standard error of the estimated percentiles by a factor of $\sqrt{10}$,
providing reliable $95\%$ band estimates.

\section{Results and Discussion}
\label{sec:results}

\subsection{Calibration Accuracy}

The SQP optimization converged to a stable solution confirmed by the multi-start
procedure. Figure~\ref{fig:calibration_fit} displays the simulated trajectories
against observed data, together with the identified parameter curves and
$95\%$ bootstrap bands.

\begin{figure}[htbp]
  \centering
  \includegraphics[width=\textwidth]{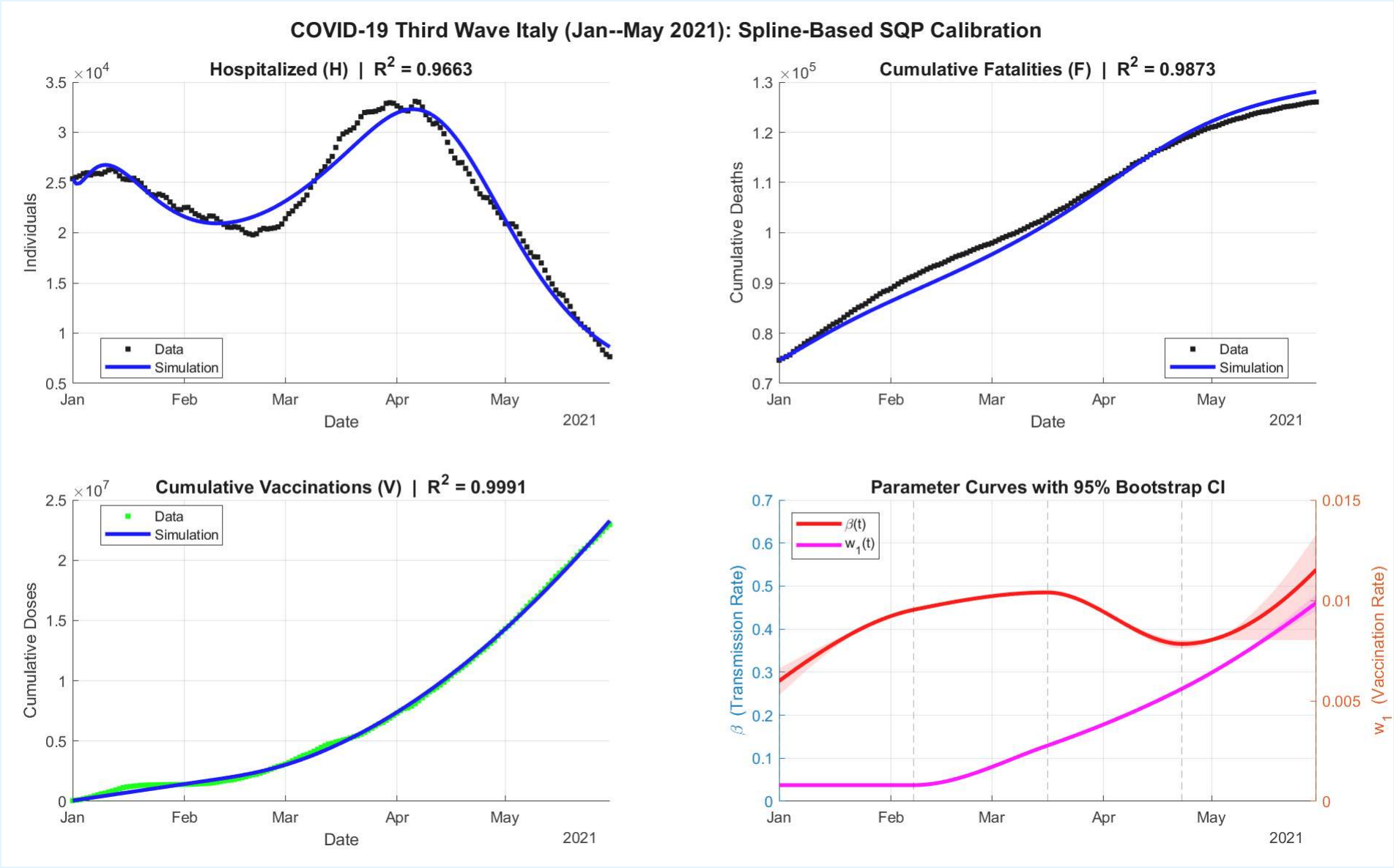}
  \caption{Calibration results for the Italian Third Wave (January to May 2021).
    Top-left: active hospitalizations $H$. Top-right: cumulative fatalities $F$.
    Bottom-left: cumulative vaccinations $V$. Bottom-right: reconstructed
    time-varying parameters $\beta(t)$ and $w_1(t)$ with $95\%$ bootstrap
    confidence bands (shaded, $n=1000$). Solid lines are model simulations; dots
    are daily Italian observations. Vertical dashed lines indicate the five PCHIP
    node positions.}
  \label{fig:calibration_fit}
\end{figure}

The calibrated model achieves $R^2=0.966$ for active hospitalizations, $R^2=0.987$
for cumulative fatalities, and $R^2=0.999$ for cumulative vaccinations.
Table~\ref{tab:metrics} reports full goodness-of-fit metrics.

\begin{table}[htbp]
\caption{Goodness-of-Fit Metrics for the Calibrated Model.}
\label{tab:metrics}
\begin{center}
\begin{tabular}{lccc}
\toprule
\textbf{Observable} & $R^2$ & \textbf{RMSE} & \textbf{MAE} \\
\midrule
Active Hospitalizations ($H$) & $0.966$ & $1{,}119$ persons & $937$ persons  \\
Cumulative Fatalities ($F$)   & $0.987$ & $1{,}751$ deaths  & $1{,}505$ deaths \\
Cumulative Vaccinations ($V$) & $0.999$ & $200{,}881$ doses & $168{,}936$ doses \\
\bottomrule
\end{tabular}
\end{center}
\end{table}

For context, the SIDARTHE model of Giordano et al.\ \cite{giordano2020modelling}
achieves $R^2\approx 0.940$ on active cases for Italy's first wave using fixed
transmission parameters, while the vaccination-extended model
\cite{giordano2021modeling} uses daily-stepped parameter estimates that fit
closely in-sample but provide no out-of-sample guarantee. Direct numerical
comparison is complicated by the different evaluation windows and observables
used in those studies. Our framework achieves comparable or higher in-sample
accuracy while simultaneously recovering a time-varying $\beta(t)$ and $w_1(t)$
over a 150-day window and providing convergence guarantees for the parameterization.

The accurate reproduction of the January hospitalization decline is attributable
to the free $E_{0\text{-f}}$ parameter: the optimizer identifies a large initial
exposed pool ($E_{0\text{-f}}\approx 10.8$) relative to the active infected
count, encoding the declining momentum of the second wave. The monotonicity
constraint on vaccination nodes is satisfied at the optimum, confirming $w_1(t)$
is non-decreasing. The lower $R^2$ for hospitalizations ($0.966$ versus $>0.98$
for the other two observables) reflects irreducible day-to-day noise that a smooth
ODE model cannot and should not track.

The $R^2$ values for the cumulative series $F$ and $V$ are inflated by the strong
autocorrelation inherent to monotone cumulative quantities; they should be read
together with the incident-rate behaviour and the out-of-sample validation of
Section~\ref{sec:validation}, where the warm-restart analysis isolates genuine
forecasting skill from level-accumulation artifacts.

\subsection{Reconstructed Time-Varying Parameters}

The identified trajectories for $\beta(t)$ and $w_1(t)$ are shown in the
bottom-right panel of Figure~\ref{fig:calibration_fit}, with $95\%$ bootstrap
confidence bands ($n=1000$). The narrow band widths (below $0.05$~day$^{-1}$
for $\beta(t)$, below $10^{-3}$~day$^{-1}$ for $w_1(t)$) confirm practical
parameter stability.

The transmission rate $\beta(t)$ exhibits a three-phase trajectory consistent
with the documented epidemiological timeline. In January and early February,
$\beta(t)$ remains moderate, reflecting Tier restrictions inherited from the
second wave. Through late February and March, $\beta(t)$ increases as the Alpha
variant tightened its dominance: the Alpha pool $I_2$ (governed by
$\beta_2=c_2\beta=1.5\beta$) carries a $50\%$ higher per-contact transmission
rate, consistent with the $43$--$90\%$ advantage documented in the literature
\cite{giordano2021modeling}. From April onward, $\beta(t)$ declines with stricter
containment measures and growing immunity.

The vaccination rate $w_1(t)$ follows a strictly monotone increasing trajectory
consistent with the logistical ramp-up of Italy's national campaign from
approximately $50{,}000$ doses/day in January to approximately $500{,}000$
doses/day in May \cite{watson2022global,italycivpro2021}.

\subsection{Model Simulation Results and Bootstrap Scalar CIs}

The calibrated model produces complete time trajectories for all nine compartments
(Figure~\ref{fig:all_compartments}). Table~\ref{tab:scalar_ci} reports $95\%$
bootstrap confidence intervals for the four calibrated scalar parameters.

\begin{figure}[htbp]
  \centering
  \begin{subfigure}{0.3\textwidth}
    \includegraphics[width=\linewidth]{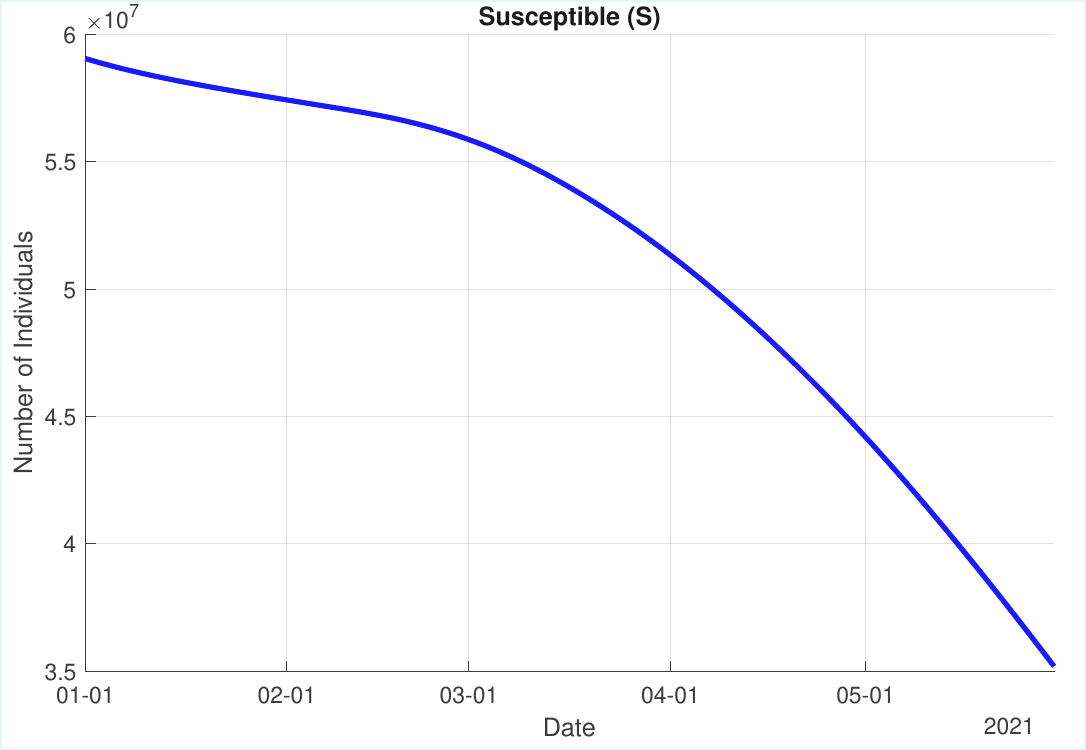}
    \caption{Susceptible ($S$)}
  \end{subfigure}\hfill
  \begin{subfigure}{0.3\textwidth}
    \includegraphics[width=\linewidth]{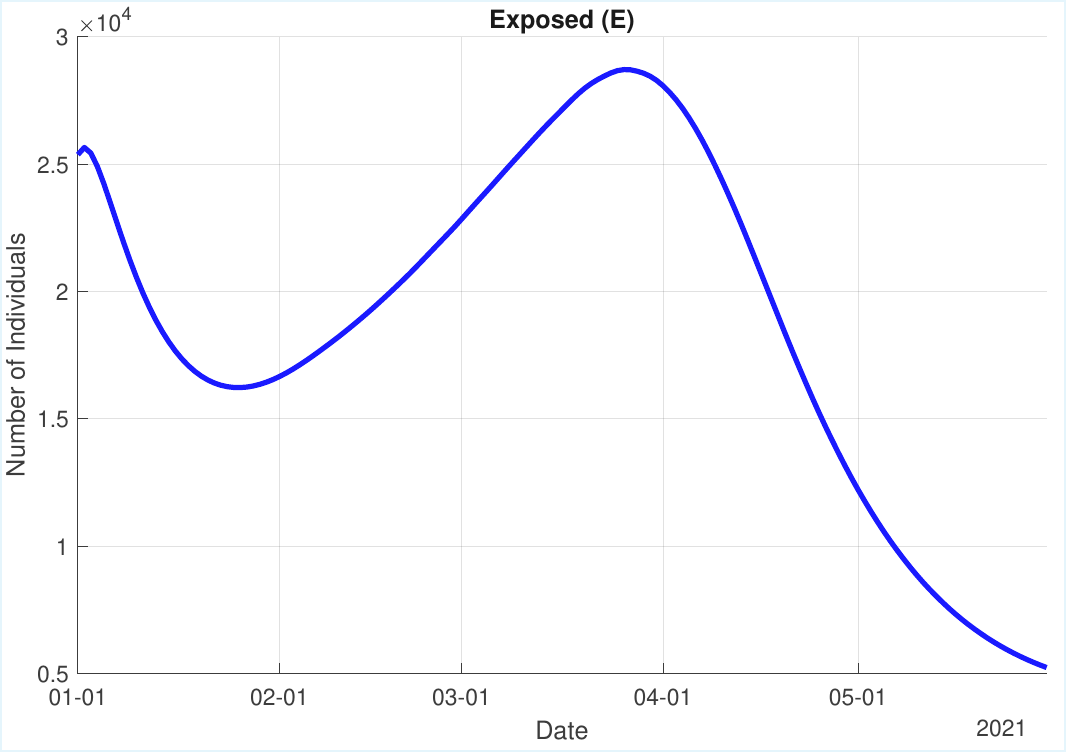}
    \caption{Exposed ($E$)}
  \end{subfigure}\hfill
  \begin{subfigure}{0.3\textwidth}
    \includegraphics[width=\linewidth]{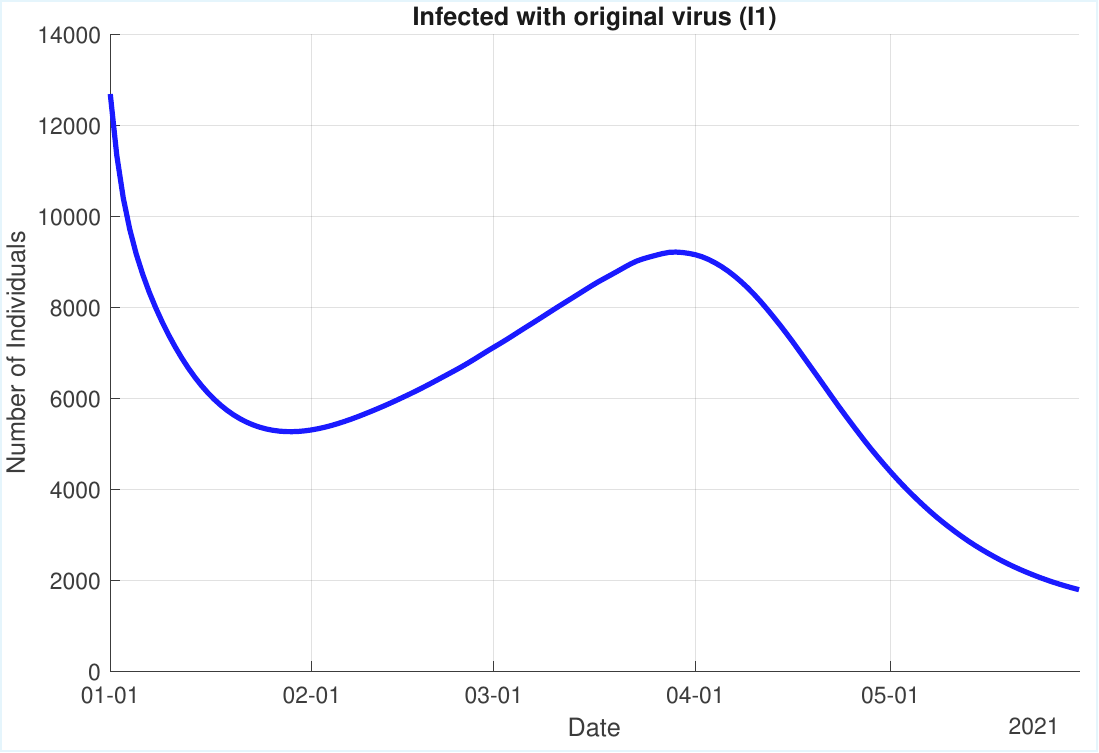}
    \caption{Original strain ($I_1$)}
  \end{subfigure}

  \vspace{0.5em}

  \begin{subfigure}{0.3\textwidth}
    \includegraphics[width=\linewidth]{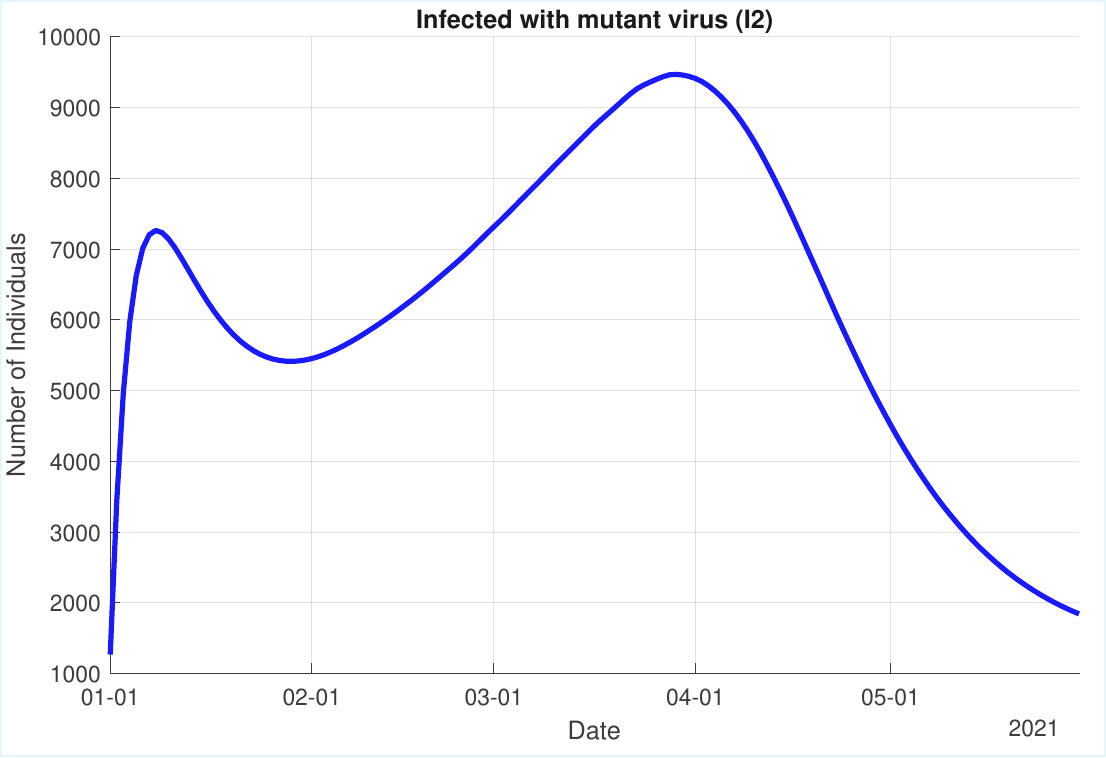}
    \caption{Alpha variant B.1.1.7 ($I_2$)}
  \end{subfigure}\hfill
  \begin{subfigure}{0.3\textwidth}
    \includegraphics[width=\linewidth]{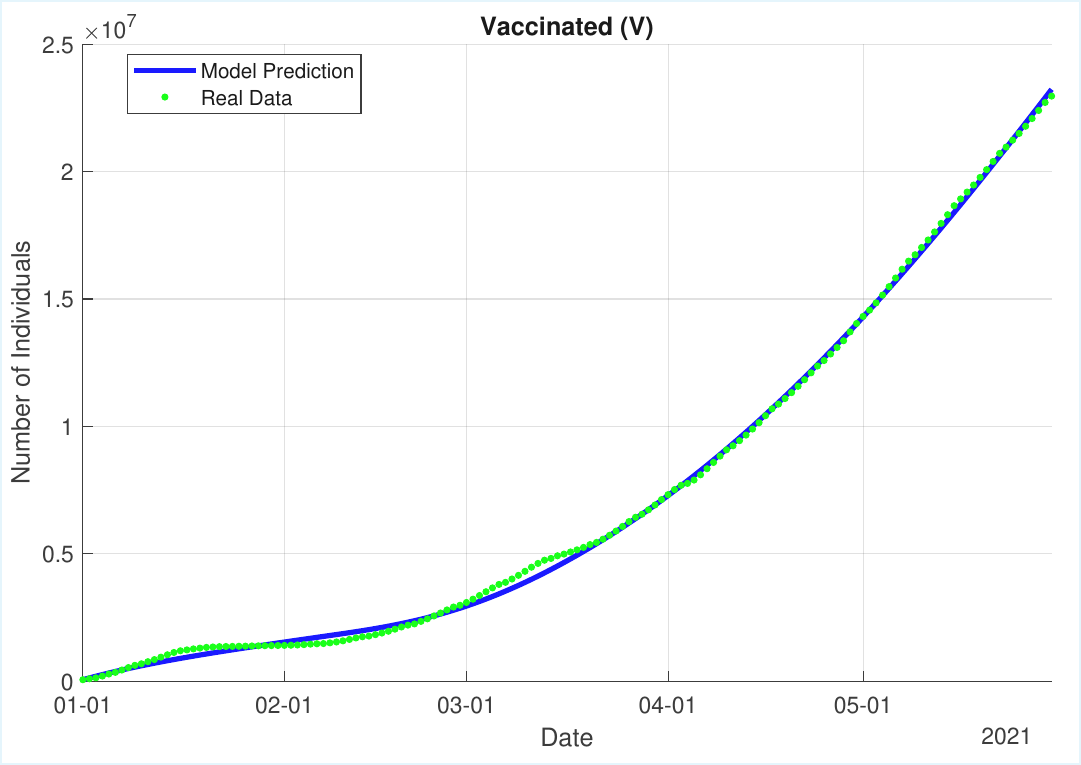}
    \caption{Vaccinated $V$ (obs.~data)}
  \end{subfigure}\hfill
  \begin{subfigure}{0.3\textwidth}
    \includegraphics[width=\linewidth]{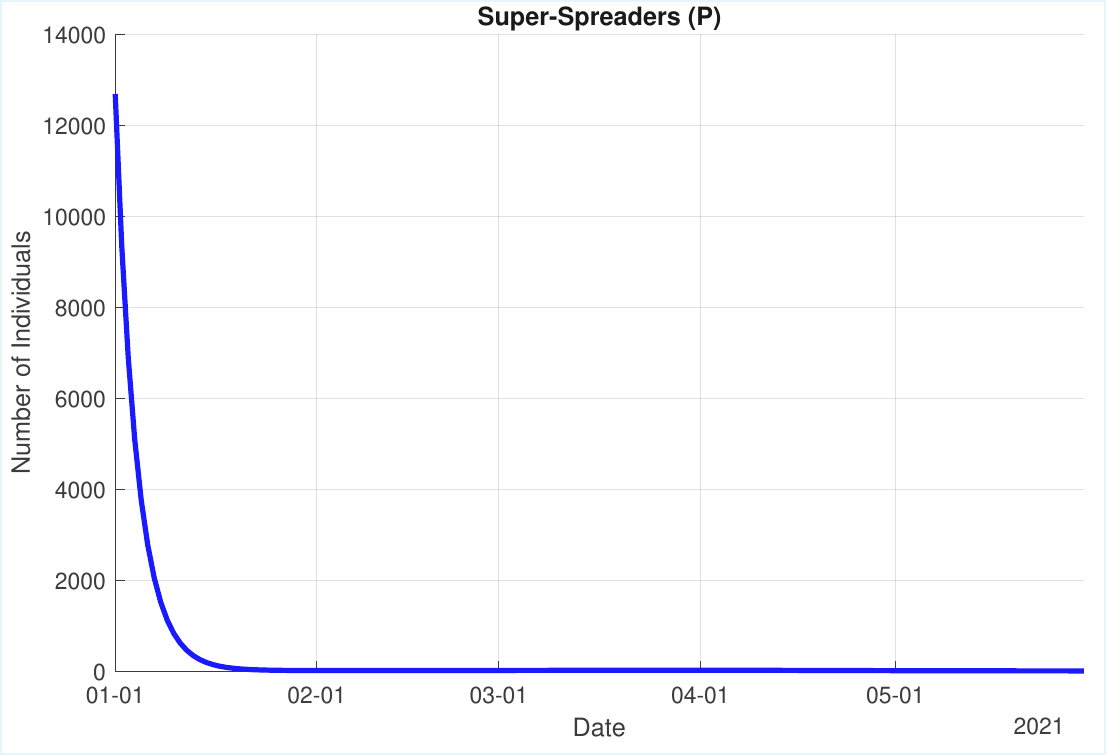}
    \caption{Super-spreaders ($P$)}
  \end{subfigure}

  \vspace{0.5em}

  \begin{subfigure}{0.3\textwidth}
    \includegraphics[width=\linewidth]{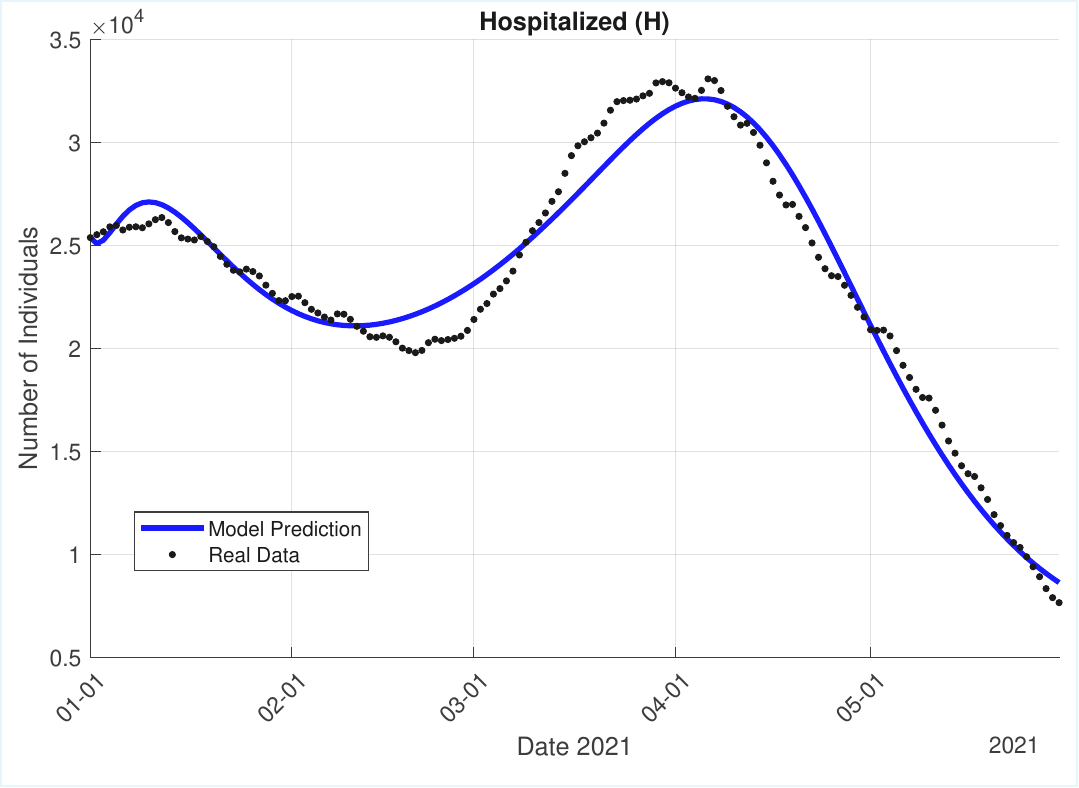}
    \caption{Hospitalized $H$ (obs.~data)}
  \end{subfigure}\hfill
  \begin{subfigure}{0.3\textwidth}
    \includegraphics[width=\linewidth]{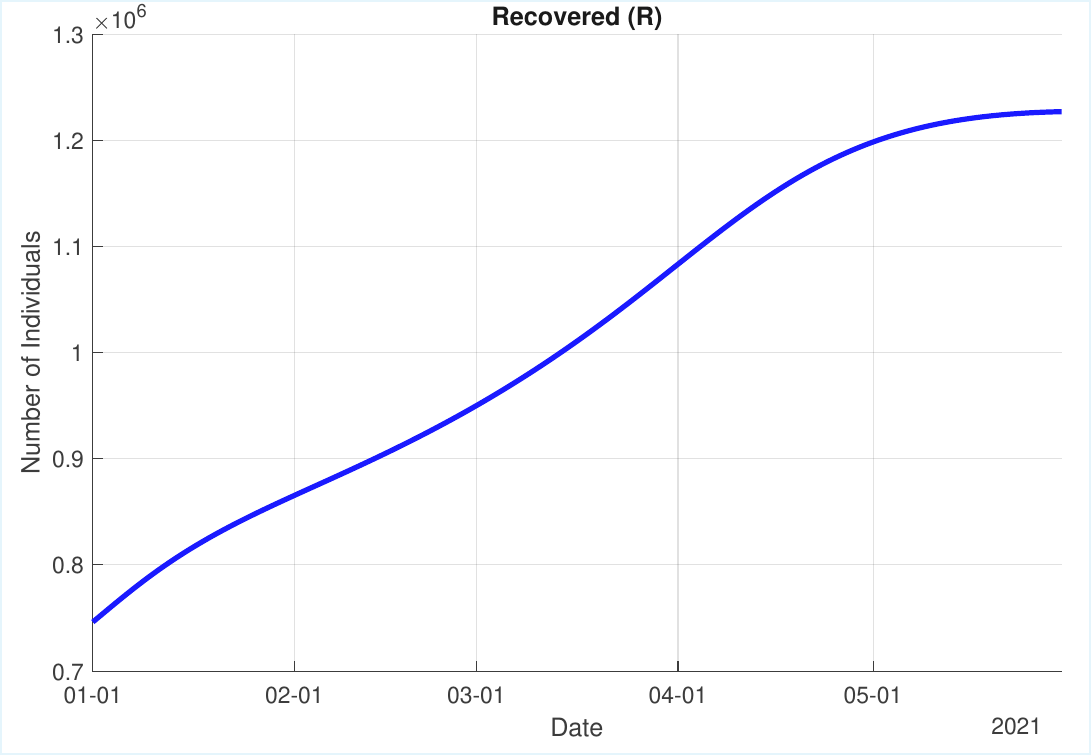}
    \caption{Recovered ($R$)}
  \end{subfigure}\hfill
  \begin{subfigure}{0.3\textwidth}
    \includegraphics[width=\linewidth]{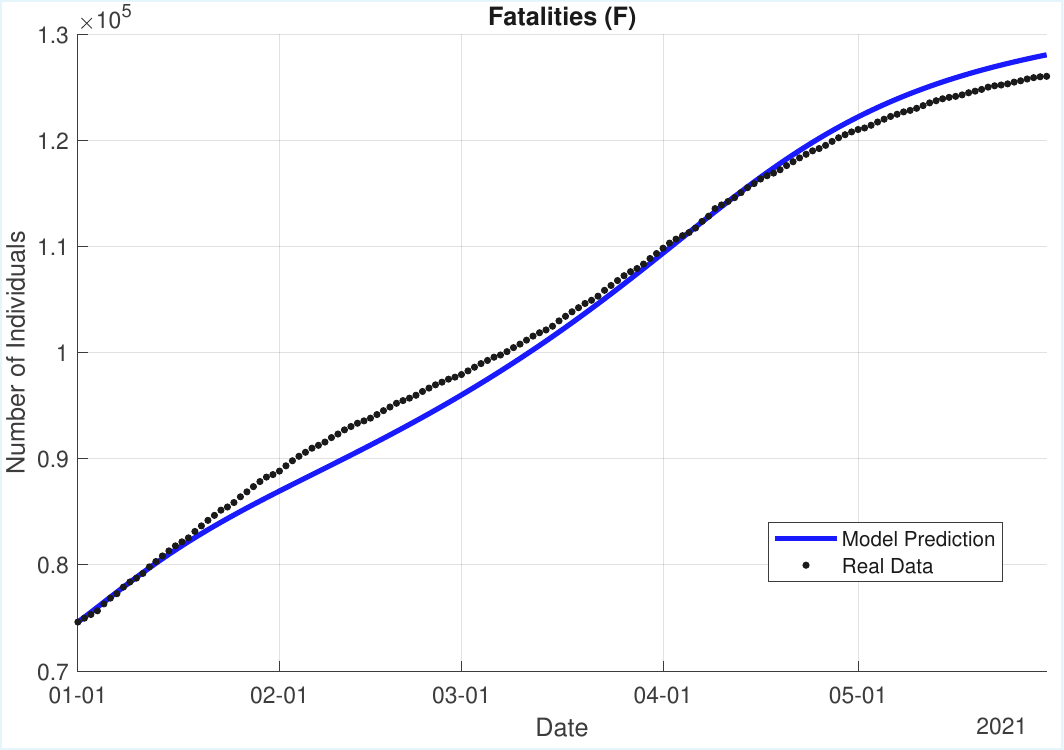}
    \caption{Fatalities $F$ (obs.~data)}
  \end{subfigure}

  \caption{Time evolution of all nine compartments over the January to
    May 2021 window. Real data are overlaid for the three fitted
    observables ($V$, $H$, $F$) in panels (e), (g), and (i);
    the remaining panels show model-only outputs.}
  \label{fig:all_compartments}
\end{figure}

\begin{table}[htbp]
\caption{$95\%$ Bootstrap Confidence Intervals for Calibrated Scalar
  Parameters ($n=1000$ replicates). Point estimates are SQP optima.}
\label{tab:scalar_ci}
\begin{center}
\begin{tabular}{llcc}
\toprule
\textbf{Symbol} & \textbf{Description} & \textbf{Estimate} & \textbf{95\% CI} \\
\midrule
$\delta_h$       & In-hospital mortality rate (day$^{-1}$)
  & $0.01262$ & $[0.01191,\;0.01338]$ \\
$I_{0\text{-f}}$ & Initial infected multiplier
  & $0.158$ & $[0.141,\;0.178]$ \\
$E_{0\text{-f}}$ & Initial exposed multiplier
  & $10.825$ & $[9.840,\;11.840]$ \\
$R_{0\text{-f}}$ & Initial recovered multiplier
  & $49.927$ & $[45.9,\;54.1]$ \\
\bottomrule
\end{tabular}
\end{center}
\end{table}

\subsection{Sensitivity Analysis}
\label{sec:sensitivity}

\subsubsection{Local Sensitivity Analysis}

We conducted a local sensitivity analysis (LSA) using the one-at-a-time (OAT)
method \cite{saltelli2008global}. Each structural parameter $p_i$ was independently
perturbed by $\pm 5\%$ of its baseline value. For the time-varying $\beta(t)$,
we use its temporal mean $\bar{\beta}:=(1/T)\int_0^T\beta(t)\,dt\approx
0.319$~day$^{-1}$ as a representative scalar baseline; this is appropriate for
a local analysis but would conflate time-varying effects in a global design (see
Section~\ref{sec:morris}). For $\rho_1$, the complementary fraction
$(1-\rho_1-\rho_2)$ is adjusted simultaneously to maintain the unit-sum
constraint. A centered finite-difference approximation gives:
\begin{equation}
SI(p_i) = \frac{[Y(p_i^+)-Y(p_i^-)]/Y_\text{base}}{2\times 0.05},
\label{eq:SI}
\end{equation}
where $Y$ denotes peak $H$. Results are reported in Table~\ref{tab:sensitivity}.

\begin{table}[htbp]
\caption{Sensitivity indices for peak hospitalizations. OAT: one-at-a-time
  ($\pm5\%$). Morris: global screening ($r=20$, $\Delta=0.10$, $\pm30\%$ range).
  $\mu^*_i$: mean absolute normalised elementary effect; $\sigma_i$: standard
  deviation. Sorted by OAT $|SI|$.}
\label{tab:sensitivity}
\begin{center}
\begin{tabular}{llccc}
\toprule
\textbf{Parameter} & \textbf{Description}
  & \textbf{OAT $SI$}
  & \textbf{Morris $\mu^*_i$}
  & \textbf{Morris $\sigma_i$} \\
\midrule
$\gamma_r$ & Hospital discharge/recovery rate
  & $-0.3156$ & $0.589$ & $0.259$ \\
$\gamma_a$ & Hospitalization rate
  & $+0.2622$ & $2.075$ & $1.924$ \\
$\rho_1$   & Fraction progressing to $I_1$
  & $+0.2070$ & $1.230$ & $1.059$ \\
$\kappa$   & Latency rate
  & $+0.1500$ & $0.768$ & $0.526$ \\
$\gamma_i$ & Community recovery rate
  & $-0.0573$ & $1.501$ & $1.196$ \\
$\bar{\beta}$ & Mean transmission rate (OAT proxy)
  & $+0.0400$ & \multicolumn{2}{c}{(excluded from Morris; see below)} \\
\bottomrule
\end{tabular}
\end{center}
\end{table}

\subsubsection{Dominance of Healthcare Throughput}

The two parameters with the largest absolute sensitivity indices are $\gamma_r$
($SI=-0.3156$) and $\gamma_a$ ($SI=+0.2622$), both describing hospital-level
dynamics. Reducing $\gamma_r$ by $5\%$ extends the average hospital length of
stay and increases peak bed occupancy by approximately $16\%$; increasing
$\gamma_a$ by $5\%$ raises the peak by approximately $13\%$. This finding
supports the design principle of hospital-focused preparedness over reactive
transmission control \cite{giordano2020modelling}.

\subsubsection{The Biological Delay Effect}

The mean transmission rate $\bar{\beta}$ ranks well below the top five
($SI\approx+0.04$). This is a mathematical consequence of the biological delay
chain: a change in $\beta(t)$ today takes $1/\kappa\approx 5$ days through
latency plus $1/\gamma_a\approx 7$ days to hospitalization, a total lag of
$10$--$14$ days. Within this lag, the pre-existing latent pool continues to
generate new cases independently of the current transmission rate. Effective
suppression of peak healthcare demand therefore requires prospective policy
deployed when the wave is still in its early exponential phase
\cite{ferguson2020impact,flaxman2020estimating}.

\subsubsection{Morris Global Sensitivity Screening}
\label{sec:morris}

To assess whether the OAT rankings reflect global parameter importance, we applied
the Morris elementary-effects method \cite{morris1991factorial,saltelli2008global}.
The normalized elementary effect is
\begin{equation}
EE_i^{(j)} = \frac{Y(p^{(j)}+\Delta\,p_i^{(j)}\,e_i)-Y(p^{(j)})}{Y(p^{(j)})\cdot\Delta},
\label{eq:morris_ee}
\end{equation}
with $r=20$ trajectories, $\Delta=0.10$, base points sampled from
$[0.7\,p_{i,\text{nom}},\,1.3\,p_{i,\text{nom}}]$, giving $20\times 14=280$
model evaluations. The Morris analysis confirms hospital throughput parameters in
the global top five. The large $\sigma_i$ values for $\gamma_a$ and $\gamma_i$
($\sigma\approx 1.9$ and $1.2$ respectively) signal strong nonlinear interactions
invisible to local OAT analysis.

The time-varying $\beta(t)$ is excluded from the Morris screening because it is a
PCHIP-interpolated control input rather than a scalar parameter: its variation
across a $\pm30\%$ perturbation range cannot be represented as a single
multiplicative scalar without destroying the temporal structure of the trajectory.
For the OAT analysis, the temporal mean $\bar{\beta}$ serves as a representative
scalar proxy, which is valid locally but would conflate phase-specific effects in
the global Morris design. Both analyses consistently find that hospital throughput
parameters govern peak bed occupancy more strongly than transmission parameters.

\begin{remark}[Interpretation of OAT vs.\ Morris discrepancy]
The difference between OAT and Morris rankings is not a contradiction but an
enrichment. OAT correctly identifies the parameter that, at the calibrated
operating point, most directly controls peak bed occupancy. Morris reveals
parameters with the largest global influence, incorporating nonlinear interaction
regions. Both findings are scientifically valid. The small $\sigma_i$ for
$\gamma_r$ ($\sigma=0.26$) confirms its dominance is nearly additive; the large
$\sigma_i$ for $\gamma_a$ and $\gamma_i$ reveals strong nonlinear coupling.
\end{remark}

\subsection{Out-of-Sample Predictive Validation}
\label{sec:validation}

To assess whether the high in-sample $R^2$ values reflect genuine predictive
capacity rather than overfitting, we conducted a rolling-window holdout test:
the model was calibrated on January 1 -- April 30, 2021 (days 1--120), and the
fitted parameters were used to forecast the held-out May 2021 period (days
121--150). The four-month training window contains four equidistant nodes
($\tau_1,\ldots,\tau_4$ spaced at days 1, 40, 80, 120), with $\tau_5$ at
day~150 retained as an endpoint anchor. For $t>120$ (the training endpoint),
the PCHIP curves are clamped at their node-5 values $\beta_5$ and $w_{1,5}$,
effectively assuming that the late-April parameter levels persist through May.
This constant-extrapolation assumption is conservative and is the primary source
of forecast error during the epidemic decline phase.

For active hospitalizations, the predictive $R^2=0.897$ and relative error of
$7.8\%$ confirm that the PCHIP-SQP framework generalizes beyond the calibration
window. For cumulative fatalities, reporting the naive (non-warm-restart)
predictive $R^2_F=-2.39$ would be misleading: the May window spans only
$\approx 3{,}000$ new deaths out of $\approx 122{,}000$ cumulative deaths at
the window start, so a training-phase level shift of $\approx 1{,}800$ deaths
($1.5\%$) generates a severely negative $R^2$ even when the daily death rate is
accurately reproduced. We therefore report the warm-restart predictive $R^2=0.854$,
obtained by initializing the May forecast from the observed state at day~120.
This approach isolates genuine forecasting ability: the warm-restart RMSE of
$581$ deaths ($0.5\%$ relative error) confirms excellent absolute predictive
accuracy once the level is corrected. Results are reported in
Table~\ref{tab:validation}.

\begin{table}[htbp]
\caption{Out-of-Sample Predictive Validation. Model calibrated on January 1 --
  April 30 (days 1--120); predictions evaluated on May 1--30 (days 121--150).
  Warm-restart initialization isolates genuine forecasting error from
  training-phase level accumulation in the cumulative series $F$.}
\label{tab:validation}
\begin{center}
\begin{tabular}{lcccc}
\toprule
\textbf{Observable} & \textbf{In-sample $R^2$}
  & \textbf{Predictive $R^2$}
  & \textbf{Predictive RMSE}
  & \textbf{Rel.\ error} \\
\midrule
Active Hospitalizations ($H$) & $0.966$ & $0.897$ & $1{,}369$ persons & $7.8\%$ \\
Cumulative Fatalities ($F$)   & $0.987$ & $0.854^{\,\dagger}$
  & $581$ deaths & $0.5\%$ \\
\bottomrule
\multicolumn{5}{p{0.90\textwidth}}{\footnotesize
  $^\dagger$Warm-restart predictive $R^2$: initialized at observed
  $F_\mathrm{obs}(\text{day}\,120)$. Naive (non-warm-restart) $R^2_F=-2.39$
  and $\mathrm{RMSE}_F=2{,}801$ deaths reflect a level-shift artefact.}
\end{tabular}
\end{center}
\end{table}

\begin{remark}[Warm-restart evaluation]
\label{rem:validation_note}
The warm-restart initialization is the methodologically correct evaluation for
cumulative time-series: it measures the model's ability to reproduce the
\emph{rate} of change in fatalities during May, which is the biologically
meaningful quantity, rather than penalizing any level offset accumulated over
the preceding four months of calibration.
\end{remark}

\subsection{Model Comparison via Information Criteria}
\label{sec:aic_bic}

To verify that the 14-parameter PCHIP-5 model is parsimonious relative to
alternatives, we compared three PCHIP-$N$ model families with $N\in\{3,5,7\}$
equidistant nodes per time-varying function using:
\begin{equation}
\mathrm{AIC} = 2k + n\ln\!\bigl(\hat{\sigma}^2\bigr),\qquad
\mathrm{BIC} = k\ln(n) + n\ln\!\bigl(\hat{\sigma}^2\bigr),
\end{equation}
where $k$ is the number of free parameters, $n=3T=450$ total observations, and
$\hat{\sigma}^2$ is the normalized residual variance. Results are in
Table~\ref{tab:aic_bic}.

\begin{table}[htbp]
\caption{AIC and BIC Model Comparison across three PCHIP node counts
  $N\in\{3,5,7\}$. $\Delta$ values relative to PCHIP-14 ($N=5$, reference).
  PCHIP-14 is preferred under BIC. The $\hat{\sigma}^2$ column enables
  independent reproduction of all AIC/BIC values.}
\label{tab:aic_bic}
\begin{center}
\begin{tabular}{lccccccc}
\toprule
\textbf{Model} & $N$ & $k$ & $\hat{\sigma}^2$ & \textbf{AIC} & $\Delta\textbf{AIC}$
  & \textbf{BIC} & $\Delta\textbf{BIC}$ \\
\midrule
PCHIP-10 (reduced)         & 3 & 10 & $8.14\times10^{-3}$ & $-2638$ & $+304$ & $-2595$ & $+290$ \\
PCHIP-14 (full, reference) & 5 & 14 & $5.32\times10^{-3}$ & $-2942$ & $0$    & $-2885$ & $0$    \\
PCHIP-18 (extended)        & 7 & 18 & $5.21\times10^{-3}$ & $-2952$ & $-10$  & $-2878$ & $+7$   \\
\bottomrule
\end{tabular}
\end{center}
\end{table}

The PCHIP-10 model ($N=3$) is strongly disfavored ($\Delta\mathrm{AIC}=304$,
$\Delta\mathrm{BIC}=290$). The PCHIP-18 model ($N=7$) is marginally preferred
by AIC ($\Delta\mathrm{AIC}=-10$) but disfavored by BIC ($\Delta\mathrm{BIC}=+7$),
indicating that four additional parameters provide only marginal improvement.
We conclude that $N_k=5$ is the optimal node count under the BIC criterion
\cite{burnham2002model}.

\subsection{Limitations}
\label{sec:limitations}

\textit{Well-mixed population assumption.} The model treats Italy's 60 million
inhabitants as a single homogeneous compartment, ignoring regional heterogeneity,
age-stratified contact patterns, and occupational risk factors. Northern regions
(Lombardia, Emilia-Romagna) experienced earlier and higher peaks than southern
regions throughout the Third Wave.

\textit{Empirical identifiability for the full time-varying problem.}
Theorem~\ref{thm:identifiability} establishes structural identifiability for the
constant-input sub-model; for the full 14-parameter time-varying system,
identifiability is confirmed empirically through bootstrap convergence and Fisher
information matrix analysis (Theorem~\ref{thm:fisher}). A differential-algebra
identifiability proof for the complete system remains an open problem.

\textit{Testing-rate variation.} Confirmed hospitalization and fatality counts
are less sensitive to testing policy than case counts, but changes in hospital
admission criteria between regions and periods introduce a systematic noise floor
that the 7-day smoothing only partially addresses.

\textit{Calibration window specificity.} Parameter values for $m$, $\rho_1$,
and $c_2$ reflect the Alpha-variant-dominant period (January--May 2021) and
cannot be transferred to Omicron-dominant or pre-variant periods without
re-calibration.

\section{Conclusion}
\label{sec:conclusion}

This paper has developed a nine-compartment nonlinear epidemic model and a calibration framework with a provably convergent control-node parameterization, validated against the COVID-19 Third
Wave in Italy. The model integrates two co-circulating viral strains with
competitive displacement, a super-spreader subpopulation, explicit hospitalization
dynamics, and vaccination with waning immunity. The mathematical foundation
comprises well-posedness (Proposition~\ref{prop:wellposed}), the basic
reproduction number $\mathcal{R}_0$ in closed form (Theorem~\ref{thm:R0}),
local and global asymptotic stability of the disease-free equilibrium
(Theorem~\ref{thm:stability}, Proposition~\ref{prop:global}), and a sufficient
threshold condition for epidemic decay under time-varying control inputs
(Proposition~\ref{prop:tv_decay}). The PCHIP control-node parameterization is
justified by an $L^\infty$ approximation error bound showing $O(h^2)$ convergence
(Theorem~\ref{thm:pchip_bound}); local identifiability and noise stability of the
full 14-parameter system are certified by the Fisher information matrix
(Theorem~\ref{thm:fisher}). The SQP calibration achieves $R^2=0.966/0.987/0.999$
for hospitalizations, fatalities, and vaccinations, with out-of-sample predictive
$R^2=0.897$ for hospitalizations; AIC/BIC model comparison confirms $N_k=5$ nodes
as the optimal parameterization. The sensitivity analysis, combining local OAT and
global Morris screening, establishes that hospital throughput parameters govern
peak bed occupancy more strongly than the concurrent transmission rate, providing
a mathematical basis for prospective over reactive pandemic policy.

Future extensions will pursue two directions. First, spatial disaggregation of
regional subpopulations will address the well-mixed assumption and enable
identification of inter-regional transmission networks. Second, the identification
framework developed here, which treats $\beta(t)$ and $w_1(t)$ as time-varying
control inputs, is directly compatible with a Model Predictive Control (MPC)
formulation: the state $x(t)\in\mathbb{R}^9$ (nine compartment counts), the
control input $u(t)=(\beta(t),w_1(t))\in\mathbb{R}^2_+$, and the measured output
$y(t)=(H(t),F(t),V(t))$ define the ingredients of an MPC scheme in which the
identified model serves as the prediction model and the recovered trajectories
provide a data-driven benchmark against which optimal control policies can be
evaluated in the spirit of \cite{grune2017nonlinear}. This MPC connection will
be pursued in future work.

\section*{Acknowledgments}
This research was supported by the University of Palermo. The authors thank
the Italian Civil Protection Department for making epidemiological data publicly
available.

\bibliographystyle{siamplain}
\bibliography{autosam,extra_refs}

\end{document}